\documentclass[12pt]{amsart}

 \usepackage{amsfonts,graphics,amsmath,amsthm,amsfonts,amscd,amssymb,amsmath,latexsym,multicol,
 mathrsfs}
\usepackage{epsfig,url}
\usepackage{flafter}
\usepackage{fancyhdr}
\usepackage{hyperref}
\hypersetup{colorlinks=true, linkcolor=black}

\makeatletter

\def\jobis#1{FF\fi
  \def\predicate{#1}%
  \edef\predicate{\expandafter\strip@prefix\meaning\predicate}%
  \edef\job{\jobname}%
  \ifx\job\predicate
}

\makeatother

\if\jobis{proposal}%
\else
\fi

 \usepackage[matrix, arrow]{xy}

\DeclareMathOperator{\Supp}{Supp}

\DeclareMathOperator{\Spec}{Spec}

\DeclareMathOperator{\lct}{lct}

 \numberwithin{equation}{subsection}
 \numberwithin{footnote}{subsection}

 \newtheorem{lem}[subsection]{Lemma}
 \newtheorem{prop}[subsection]{Proposition}
 \newtheorem{thm}[subsection]{Theorem}

{
    \newtheoremstyle{upright}%
        {8pt plus2pt minus4pt}%
        {8pt plus2pt minus4pt}%
        {\upshape}%
        {}%
        {\bfseries\scshape}%
        {}%
        {1em}%
        {}%
\theoremstyle{upright}

 \newtheorem{constr}[subsection]{Construction}

 \newtheorem{rem}[subsection]{Remark}

}

 \newcommand{\N}{\mathbb N}
 \newcommand{\PP}{\mathbb P}
 
 \newcommand{\Q}{\mathbb Q}
 \newcommand{\R}{\mathbb R}
 \newcommand{\Z}{\mathbb Z}
 \newcommand{\bir}{\dashrightarrow}
 \newcommand{\rddown}[1]{\left\lfloor{#1}\right\rfloor} 

\title{\large E\MakeLowercase{xistence of flips and minimal models for 3-folds in char $p$}}
\author{\large C\MakeLowercase{aucher} B\MakeLowercase{irkar}}
\date{\today}
\begin{document}
\maketitle

\begin{abstract}
We will prove the following results for $3$-fold pairs $(X,B)$  
over an algebraically closed field $k$ of characteristic $p>5$: 
log flips exist for $\Q$-factorial dlt pairs $(X,B)$; 
log minimal models exist for projective klt pairs $(X,B)$ with pseudo-effective $K_X+B$;
the log canonical ring $R(K_X+B)$ is finitely generated for projective klt pairs $(X,B)$ when $K_X+B$ is a 
big $\Q$-divisor; semi-ampleness holds for a nef and big $\Q$-divisor $D$ if $D-(K_X+B)$ is nef and big 
and $(X,B)$ is projective klt;  
$\Q$-factorial dlt models exist for lc pairs $(X,B)$; terminal models exist for klt pairs $(X,B)$; 
 ACC holds for lc thresholds; etc.   
\end{abstract}

\tableofcontents


\section{Introduction}

We work over an algebraically closed field $k$ of characteristic (char) $p>0$.
The pairs $(X,B)$ we consider in this paper always have $\R$-boundaries $B$ unless otherwise 
stated. 

Higher dimensional birational geometry in char $p$ is still largely conjectural. Even the most basic problems 
such as base point freeness are not solved in general. 
Ironically though Mori's work on existence of rational curves which plays an 
important role in characteristic $0$ uses reduction mod $p$ techniques. 
There are two reasons among others which have held back progress in char $p$: 
resolution of singularities is not known and Kawamata-Viehweg vanishing fails. 
However, it was expected that one can work 
out most components of the minimal model program in dimension $3$. This is because 
resolution of singularities is known in dimension $3$ and many problems can be 
reduced to dimension $2$ hence one can use special features of surface geometry. 

On the positive side there has been some good progress toward understanding birational 
geometry in char $p$. People have tried to replace the characteristic $0$ tools that fail 
in char $p$. For example, 
Keel [\ref{Keel}] developed techniques for dealing with the base point free problem and 
semi-ampleness questions in general without relying on Kawamata-Viehweg 
vanishing type theorems. On the other hand, motivated by questions in commutative 
algebra, people have introduced Frobenius-singularities whose definition do not 
require resolution of singularities and they are very similar to singularities in 
characteristic $0$ (cf. [\ref{Schwede}]).

More recently Hacon-Xu [\ref{HX}] proved the existence 
of flips in dimension $3$ for pairs $(X,B)$ with $B$ having standard coefficients, that is, coefficients in 
$\mathfrak{S}=\{1-\frac{1}{n}\mid n\in \N\cup \{\infty\}\}$, and char $p>5$. From this they could derive 
existence of minimal models for $3$-folds with canonical singularities. 
In this paper, we rely on their results and ideas.
The requirement $p>5$ has to do with the behavior of singularities on surfaces, eg a klt surface singularity 
over $k$ of char $p>5$ is strongly $F$-regular.\\
 
{\textbf{\sffamily{Log flips.}}} Our first result is on the existence of flips. 

\begin{thm}\label{t-flip-1}
Let $(X,B)$ be a $\Q$-factorial dlt pair of dimension $3$ over $k$ of char $p>5$.  
Let $X\to Z$ be a $K_X+B$-negative extremal flipping projective contraction. Then its flip exists. 
\end{thm}

The conclusion also holds if $(X,B)$ is klt but not necessarily $\Q$-factorial. This 
follows from the finite generation below (\ref{t-fg}).
The theorem is proved in Section \ref{s-flips} when $X$ is projective. The quasi-projective case 
is proved in Section \ref{s-mmodels}.
We reduce the theorem to the case when $X$ is projective, $B$ has standard coefficients, 
and some component of $\rddown{B}$ is negative on the extremal ray: 
this case is [\ref{HX}, Theorem 4.12] which is one of the main results of that paper. 
A different approach is taken in [\ref{CGS}] to prove \ref{t-flip-1} 
 when $B$ has hyperstandard coefficients and $p\gg 0$ (these coefficients are of 
the form $\frac{n-1}{n}+\sum \frac{l_ib_i}{n}$ where $n \in\N\cup \{\infty\}$, $l_i\in \Z^{\ge 0}$ 
and $b_i$ are in some fixed DCC set).

To prove Theorem \ref{t-flip-1} we actually first prove the existence of \emph{generalized 
flips} [\ref{HX}, after Theorem 5.6]. See Section \ref{s-flips} for more details.\\

{\textbf{\sffamily{Log minimal models.}}}
In [\ref{HX}, after Theorem 5.6], using generalized flips, a \emph{generalized LMMP} is defined which is 
used to show the existence of minimal models for varieties with canonical singularities 
(or for pairs with canonical singularities and "good" boundaries). Using weak Zariski decompositions 
as in [\ref{B-WZD}], we construct log minimal models for klt pairs in general. 

\begin{thm}\label{t-mmodel}
Let $(X,B)$ be a klt pair of dimension $3$ over $k$ of char $p>5$ and let 
$X\to Z$ be a projective contraction. If $K_X+B$ is pseudo-effective$/Z$, then $(X,B)$ has a log minimal model over $Z$.
\end{thm}

The theorem is proved in Section \ref{s-mmodels}. 
Alternatively, one can apply the methods of [\ref{B-mmodel}] to construct log minimal models 
for lc pairs $(X,B)$ such that $K_X+B\equiv M/Z$ for some $M\ge 0$.
Note that when $X\to Z$ is a semi-stable fibration over a curve and $B=0$, the theorem was 
proved much earlier by Kawamata [\ref{Kawamata}].\\

{\textbf{\sffamily{Remark on Mori fibre spaces.}}} 
Let $(X,B)$ be a projective klt pair of dimension $3$ over $k$ of char $p>5$ such that $K_X+B$ is not 
pseudo-effective. An important question is whether $(X,B)$ has a Mori fibre space. 
There is an ample $\R$-divisor $A\ge 0$ such that 
$K_X+B+A$ is pseudo-effective but $K_X+B+(1-\epsilon)A$ 
is not pseudo-effective for any $\epsilon>0$. Moreover, 
we may assume that $(X,B+A)$ is klt as well (\ref{l-ample-dlt}). By Theorem \ref{t-mmodel}, 
$(X,B+A)$ has a log minimal model $(Y,B_Y+A_Y)$. Since 
$K_Y+B_Y+A_Y$ is not big, $K_Y+B_Y+A_Y$ is numerically trivial on some covering 
family of curves by [\ref{CTX}](see also \ref{t-aug-b-non-big} below). 
Again by [\ref{CTX}], there is a nef reduction map $Y\bir T$ for $K_Y+B_Y+A_Y$
which is projective over the generic point of $T$. Although $Y\bir T$ is not necessarily a 
Mori fibre space but in some sense it is similar.\\

{\textbf{\sffamily{Finite generation, base point freeness, and contractions.}}} 
We will prove finite generation in the big case from which we can derive base point freeness 
and contractions of extremal rays in many cases. These are proved in Section \ref{s-fg}.

\begin{thm}\label{t-fg}
Let $(X,B)$ be a klt pair of dimension $3$ over $k$ of char $p>5$ and $X\to Z$ a projective contraction. 
Assume that $K_X+B$ is a $\Q$-divisor which is big$/Z$. 
Then the relative log canonical algebra $\mathcal{R}(K_X+B/Z)$ is finitely generated over $\mathcal{O}_Z$. 
\end{thm}

Assume that $Z$ is a point. 
If $K_X+B$ is not big, then ${R}(K_X+B/Z)$ is still finitely generated if 
$\kappa(K_X+B)\le 1$. 
It remains to show the finite generation when $\kappa(K_X+B)=2$: this can probably be 
reduced to dimension $2$ using an appropriate canonical bundle formula, for example as 
in [\ref{CTX}].

A more or less immediate consequence of the above finite generation is the following base point 
freeness.

\begin{thm}\label{t-bpf}
Let $(X,B)$ be a projective klt pair of dimension $3$ over $k$ of char $p>5$ and $X\to Z$ a projective contraction 
where $B$ is a $\Q$-divisor. Assume that $D$ is a $\Q$-divisor such that 
$D$ and $D-(K_X+B)$ are both nef and big$/Z$. 
Then $D$ is semi-ample$/Z$.  
\end{thm}

Assume that $Z$ is a point. When $D-(K_X+B)$ is nef and big but $D$ is nef with numerical 
dimension $\nu(D)$ one or two, 
semi-ampleness of $D$ is proved in [\ref{CTX}] under some restrictions on the coefficients.

\begin{thm}\label{t-contraction}
Let $(X,B)$ be a projective $\Q$-factorial dlt pair of dimension $3$ over $k$ of char $p>5$, and 
 $X\to Z$ a projective contraction. Let $R$ be a $K_X+B$-negative extremal ray$/Z$. 
Assume that there is a nef and big$/Z$ $\Q$-divisor $N$ such that $N\cdot R=0$. 
Then $R$ can be contracted by a projective morphism. 
\end{thm}

Note that if $K_X+B$ is pseudo-effective$/Z$, then for every $K_X+B$-negative extremal 
ray $R/Z$ there exists $N$ as in the theorem (see \ref{ss-ext-rays-II}). Therefore such extremal rays 
can be contracted by projective morphisms.

 Theorems  \ref{t-bpf} and \ref{t-contraction} have been proved by Xu [\ref{Xu}]   
independently and more or less at the same time but using a different approach. His proof also relies on our 
results on flips and minimal models.\\

{\textbf{\sffamily{Dlt and terminal models.}}}
The next two results are standard consequences of the LMMP (more precisely, of special termination).
They are proved in Section \ref{s-crepant-models}.

\begin{thm}\label{cor-dlt-model}
Let $(X,B)$ be an lc pair of dimension $3$ over $k$ of char $p>5$. 
Then $(X,B)$ has a (crepant) $\Q$-factorial dlt model. In particular, if $(X,B)$ 
is klt, then $X$ has a $\Q$-factorialization by a small morphism.
\end{thm}

The theorem was proved in [\ref{HX}, Theorem 6.1] for pairs with standard coefficients.

\begin{thm}\label{cor-terminal-model}
Let $(X,B)$ be a  klt pair of dimension $3$ over $k$ of char $p>5$. 
Then $(X,B)$ has a (crepant) $\Q$-factorial terminal model.
\end{thm}

The theorem was proved in [\ref{HX}, Theorem 6.1] for pairs with standard coefficients and 
canonical singularities.\\

{\textbf{\sffamily{The connectedness principle with applications to semi-ampleness.}}}
The next result concerns the Koll\'ar-Shokurov connectedness principle. 
In characteristic $0$, the surface case was proved by Shokurov by taking a resolution and 
then calculating intersection numbers [\ref{Shokurov}, Lemma 5.7]  
but the higher dimensional case was proved by Koll\'ar by deriving it from the Kawamata-Viehweg vanishing 
theorem [\ref{Kollar+}, Theorem 17.4]. 

\begin{thm}\label{t-connectedness-d-3}
Let $(X,B)$ be a projective $\Q$-factorial pair of dimension $3$ over $k$ of char $p>5$. 
Let $f\colon X\to Z$ be a birational contraction such that 
$-(K_X+B)$ is ample$/Z$. Then for any closed point $z\in Z$, the non-klt locus of 
$(X,B)$ is connected in any neighborhood of the fibre $X_z$.
\end{thm}

The theorem is proved in Section \ref{s-connectedness}. 
To prove it we use the LMMP rather than vanishing theorems. 
When $\dim X=2$, the theorem holds in a stronger form (see \ref{t-connectedness-d-2}). 

We will use the connectedness principle on surfaces to prove some semi-ampleness results 
on surfaces and $3$-folds. 
Here is one of them:

\begin{thm}\label{t-sa-reduced-boundary}
Let $(X,B+A)$ be a projective $\Q$-factorial dlt pair of dimension $3$ over $k$ of 
char $p>5$. Assume that $A,B\ge 0$ are $\Q$-divisors such that $A$ is ample and 
$(K_X+B+A)|_{\rddown{B}}$ is nef. Then $(K_X+B+A)|_{\rddown{B}}$ is semi-ample.
\end{thm}

Note that if one could show that $\rddown{B}$ is semi-lc, then the result would follow from Tanaka
[\ref{Tanaka-2}].  In order to show 
that $\rddown{B}$ is semi-lc one needs to check that it satisfies the Serre condition $S_2$. 
In characteristic $0$ this is a consequence of Kawamata-Viehweg vanishing (see Koll\'ar [\ref{Kollar+}, Corollary 17.5]).
The $S_2$ condition can be used to glue sections on the various irreducible components of $\rddown{B}$. 
To prove the above semi-ampleness we instead use a result of Keel [\ref{Keel}, Corollary 2.9] to glue sections.\\

{\textbf{\sffamily{Log canonical thresholds.}}}
As in characteristic $0$, we will derive the following result from existence of 
$\Q$-factorial dlt models and boundedness results on Fano surfaces. 

\begin{thm}\label{t-ACC}
Suppose that $\Lambda\subseteq [0,1]$ and $\Gamma\subseteq
\mathbb{R}$ are DCC sets. Then the set
$$\{\lct(M,X,B)|\mbox{ $(X,B)$ is lc of dimension $\le 3$}\}$$

{\flushleft satisfies} the  ACC where $X$ is over $k$ with char $p>5$, 
the coefficients of $B$ belong to $\Lambda$, $M\ge 0$ is an 
$\R$-Cartier divisor with coefficients in $\Gamma$,  and $\lct(M,X,B)$ is the lc 
threshold of $M$ with respect to $(X,B)$.
\end{thm}

With some work it seems that using the above ACC one can actually prove termination for those lc pairs 
$(X,B)$ of dimension $3$ such that $K_X+B\equiv M$ for some $M\ge 0$ following the ideas in [\ref{B-acc}].
But we will not pursue this here.\\

{\textbf{\sffamily{Numerically trivial family of curves in the non-big case.}}}
We will also give a somewhat different proof of the following result which was proved 
by Cascini-Tanaka-Xu [\ref{CTX}] in char $p$. This was also proved independently by M$^{\rm c}$Kernan much earlier but unpublished. 
He informed us that his proof was inspired by [\ref{KMM}]. 

\begin{thm}\label{t-aug-b-non-big}
Assume that $X$ is a normal projective variety of dimension $d$ over an algebraically 
closed field (of any characteristic), 
and that $B,A\ge 0$ are $\R$-divisors. 
Moreover, suppose $A$ is nef and big and $D=K_X+B+A$ is nef. If $D^d=0$, then  
for each general closed point $x\in X$ there is a rational curve $L_x$ passing through 
$x$ with $D\cdot L_x=0$. 
\end{thm}

The theorem is independent of the rest of this paper. Its proof is an application of the 
bend and break theorem.\\

{\textbf{\sffamily{Some remarks about this paper.}}}
In writing this paper we have tried to give as much details as possible even 
if the arguments are very similar to the characteristic $0$ case. This is for convenience,  
future reference, and to avoid any unpleasant surprise having to do with positive characteristic. 
The main results are proved in the following order: \ref{t-flip-1} in the projective case, 
\ref{cor-dlt-model}, \ref{cor-terminal-model}, \ref{t-mmodel}, \ref{t-flip-1} in general, 
\ref{t-connectedness-d-3}, 
\ref{t-sa-reduced-boundary}, \ref{t-fg}, \ref{t-bpf}, \ref{t-contraction}, \ref{t-ACC}, and
\ref{t-aug-b-non-big}.\\

{\textbf{\sffamily{Acknowledgements.}}}
Part of this work was done when I visited National Taiwan University in September 2013 with the support of the 
Mathematics Division (Taipei Office) of the National Center for Theoretical Sciences. The visit was arranged by Jungkai A. Chen. 
I would like to thank them for their hospitality. 
This work was partially supported by a Leverhulme grant.
I would also like to thank Paolo Cascini,  Christopher Hacon, Janos Koll\'ar, James M$^{\rm{c}}$Kernan, 
Burt Totaro and Chenyang Xu for their 
comments and suggestions. Finally I would like to thank the referee for valuable corrections 
and suggestions.

\section{Preliminaries}

We work over an algebraically closed field $k$ of characteristic $p>0$ fixed throughout the paper 
unless stated otherwise.

\subsection{Contractions}\label{ss-contractions} 
A \emph{contraction} $f\colon X\to Z$ of algebraic spaces over $k$ is a proper morphism such that 
$f_*\mathcal{O}_X=\mathcal{O}_Z$. 
When  $X,Z$ are quasi-projective varieties over $k$ and $f$ is 
projective, we refer to $f$ as a \emph{projective contraction} to avoid confusion.

Let $f\colon X\to Z$ be a projective contraction of normal varieties.
We say $f$ is \emph{extremal} if the relative Kleiman-Mori cone of curves $\overline{NE}(X/Z)$ is one-dimensional.
Such a contraction is a \emph{divisorial contraction} if it is birational and it contracts
some divisor. It is called a \emph{small contraction} if it is birational and it 
contracts some subvariety of codimension $\ge 2$ but no divisors. 

Let $f\colon X\to Z$ be a small contraction and  
$D$ an $\R$-Cartier divisor such that $-D$ is ample$/Z$. We refer to $f$ as a 
\emph{$D$-flipping contraction} or just a flipping contraction for short.
We say the \emph{$D$-flip} of $f$
exists if there is a small contraction $X^+\to Z$ such that the birational transform 
$D^+$ is ample$/Z$.

\subsection{Some notions related to divisors}\label{ss-divisors}
Let $X$ be a normal projective variety over $k$ and $L$ a nef $\R$-Cartier 
divisor. We define $L^\perp:=\{\alpha\in \overline{NE}(X) \mid L\cdot \alpha=0\}$. 
This is an extremal face of $\overline{NE}(X)$ cut out by $L$.

Let $f\colon X\to Z$ be a projective morphism of normal varieties over $k$, and 
let $D$ be an $\R$-divisor on $X$. We define the algebra of $D$ over $Z$ 
as $\mathcal{R}(D/Z)=\bigoplus_{m\in\Z^{\ge 0}} f_*\mathcal{O}_X(\rddown{mD})$.
When $Z$ is a point we denote the algebra by $R(D)$.
When $D=K_X+B$ for a pair $(X,B)$ we call the algebra the \emph{log canonical algebra} 
of $(X,B)$ over $Z$.

Now let $\phi\colon X\bir Y$ be a birational map of normal projective varieties over $k$ whose inverse does 
not contract divisors. Let $D$ be an $\R$-Cartier divisor on $X$ such that $D_Y:=\phi_*D$ is 
$\R$-Cartier too. We say that $\phi$ is 
\emph{$D$-negative} if there is a common resolution $f\colon W\to X$ and $g\colon W\to Y$ 
such that $f^*D-g^*D_Y$ is effective and exceptional$/Y$, and its support contains the birational transform of 
all the prime divisors on $X$ which are contracted$/Y$.

\subsection{The negativity lemma}\label{ss-negativity}
The negativity lemma states that if $f\colon Y\to X$ is a projective birational 
contraction of normal quasi-projective varieties over $k$ and $D$ is an $\R$-Cartier 
divisor on $Y$ such that $-D$ is nef$/X$ and $f_*D\ge 0$, then $D\ge 0$ 
(since this is a local statement over $X$, it also holds if we assume 
$X$ is an algebraic space and $f$ is proper). See 
[\ref{Shokurov}, Lemma 1.1] for the characteristic $0$ case. The proof there also works 
in char $p>0$ and we reproduce it for convenience. 
Assume that the lemma does not hold. We reduce the problem to the surface case.  
Let $P$ be the image of the negative components of $D$. If  $\dim P>0$, 
we take a general hypersurface section $H$ on $X$, let $G$ be the normalization 
of the birational transform of $H$ on $Y$ and reduce the problem to the contraction 
$G\to H$ and the divisor $D|_G$. But if $\dim X>2$ and $\dim P=0$, we take a general hypersurface 
section $G$ on $Y$, let $H$ be the normalization of $f(G)$, and reduce 
the problem to the induced contraction $G\to H$ and divisor $D|_G$. So we can reduce the 
problem to the case when $X,Y$ are surfaces, $P$ is just one point, and $f$ is an isomorphism 
over $X\setminus \{P\}$. 
Taking a resolution enables us to assume $Y$ is smooth. Now let $E\ge 0$ be a divisor whose support is 
equal to the exceptional locus of $f$ and such that $-E$ is nef$/X$: pick a Cartier divisor $L\ge 0$ 
 passing through $P$ and write $f^*L=L^\sim+E$ where $L^\sim$ is the birational 
transform of $L$; then $E$ satisfies the requirements. Let $e$ be the smallest 
number such that $D+eE\ge 0$. Now there is a component $C$ of $E$ whose coefficient 
in $D+eE$ is zero and that $C$ intersects $\Supp (D+eE)$. But then 
$(D+eE)\cdot C>0$, a contradiction.

\subsection{Resolution of singularities}\label{ss-resolution}
Let $X$ be a quasi-projective variety of dimension $\le 3$ over $k$ and $P\subset X$ 
a closed subset. Assume that 
there is an open set $U\subset X$ such that $P\cap U$ is a divisor 
with simple normal crossing (snc) singularities. Then there is a \emph{log resolution} 
of $X,P$ which is an isomorphism over $U$, that is, there is a projective 
birational morphism $f\colon Y\to X$ such that the union of the exceptional locus of $f$  
and the birational transform of $P$ is an snc divisor, and   
$f$ is an isomorphism over $U$. This follows from Cutkosky [\ref{Cut-sing}, Theorems 1.1, 1.2, 1.3] 
when $k$ has char $p>5$, and from Cossart-Piltant [\ref{CP-sing}, Theorems 4.1, 4.2][\ref{CP-sing-2}, Theorem] 
in general (see also [\ref{HX}, Theorem 2.1]).

\subsection{Pairs}\label{ss-pairs} 
A \emph{pair} $(X,B)$ consists of a normal quasi-projective variety $X$ over $k$  
and an \emph{$\R$-boundary} $B$, that is an $\R$-divisor $B$ on $X$ with coefficients in $[0,1]$, 
such that $K_X+B$ is $\mathbb{R}$-Cartier. When $B$ has rational coefficients we 
say $B$ is a \emph{$\Q$-boundary} or say $B$ is \emph{rational}.
We say that $(X,B)$ is \emph{log smooth} if $X$ is smooth 
and $\Supp B$ has simple normal crossing singularities.

Let $(X,B)$ be a pair. For a prime divisor $D$ on some birational model of $X$ with a
nonempty centre on $X$, $a(D,X,B)$ denotes the \emph{log discrepancy} which is defined 
by taking a projective birational morphism $f\colon Y\to X$ from a normal variety 
containing $D$ as a prime divisor and putting $a(D,X,B)=1-b$ where $b$ is the 
coefficient of $D$ in $B_Y$ and $K_Y+B_Y=f^*(K_X+B)$.  

 As in characteristic $0$, we can define various types of singularities using 
log discrepancies. Let $(X,B)$ be a pair.
We say that the pair is \emph{log canonical} or lc for short (resp. \emph{Kawamata log terminal} or {klt} 
for short) if 
$a(D,X,B)\ge 0$ (resp. $a(D,X,B)>0$) for any prime divisor $D$ on birational models of $X$. 
An \emph{lc centre} of $(X,B)$ is the image in $X$ of a $D$ with $a(D,X,B)=0$.
The pair $(X,B)$ is \emph{terminal} if $a(D,X,B)>1$ for any prime divisor $D$ on birational models of $X$ which is 
exceptional$/X$ (such pairs are sometime called terminal in codimension $\ge 2$).
On the other hand, we say that $(X,B)$ is \emph{dlt} if there is a closed subset $P\subset X$ 
such that $(X,B)$ is log smooth 
outside $P$ and no lc centre of $(X,B)$ is inside $P$. In particular, 
the lc centres of $(X,B)$ are exactly the components of $S_1\cap \cdots \cap S_r$ where $S_i$ 
are among the components of $\rddown{B}$. Moreover, there is a log resolution 
$f\colon Y\to X$ of $(X,B)$ such that $a(D,X,B)>0$ for any prime divisor $D$ on $Y$ 
which is exceptional$/X$, eg take a log resolution $f$ 
which is an isomorphism over $X\setminus P$.  
Finally, we say that $(X,B)$ is \emph{plt} if it is dlt and each connected component of 
$\rddown{B}$ is irreducible. In particular, the only lc centres of $(X,B)$ are 
the components of $\rddown{B}$.

\subsection{Ample divisors on log smooth pairs}\label{ss-log-smooth}
Let $(X,B)$ be a projective log smooth pair over $k$ and let $A$ be an ample $\Q$-divisor. 
We will argue that there is $A'\sim_\Q A$ such that $A'\ge 0$ and that 
$(X,B+A')$ is log smooth. The argument was suggested to us by several people 
independently. We may assume that $B$ is reduced. Let $S_1,\dots, S_r$ 
be the components of $B$ and let $\mathcal{S}$ be the set of the components 
of $S_{i_1}\cap \cdots \cap S_{i_n}$ for all the choices 
$\{i_1,\dots,i_n\}\subseteq  \{1,\cdots,r\}$. By Bertini's theorem, 
there is a sufficiently divisible integer $l>0$ such that for any $T\in\mathcal{S}$, a general element of 
$|lA|_T|$ is smooth. Since $lA$ is sufficiently ample, such general elements 
are restrictions of general elements of $|lA|$. Therefore, we can choose a 
general $G\sim lA$ such that $G$ is smooth and $G|_T$ is smooth for any $T\in \mathcal{S}$. 
This means that $(X,B+G)$ is log smooth. Now let $A'=\frac{1}{l}G$.

\subsection{Models of pairs}\label{ss-mmodels} 
Let $(X,B)$ be a pair and $X\to Z$ a projective contraction over $k$.
A pair $(Y,B_Y)$ with a projective contraction $Y\to Z$ and a birational map
$\phi\colon X\bir Y/Z$ is a \emph{log birational model} of $(X,B)$ 
if  $B_Y$ is the sum of the birational transform of $B$ 
and the reduced exceptional divisor of $\phi^{-1}$. 
We say that $(Y,B_Y)$ is a \emph{weak lc model} of $(X,B)$ over $Z$ if in addition\\\\
(1) $K_Y+B_Y$ is nef/$Z$.\\
(2) for any prime divisor $D$ on $X$ which is exceptional/$Y$, we have
$$
a(D,X,B)\le a(D,Y,B_Y)
$$

And  we call $(Y,B_Y)$ a \emph{log minimal model} of $(X,B)$ over $Z$ if in addition\\\\
(3) $(Y,B_Y)$ is $\Q$-factorial dlt,\\
(4) the inequality in (2) is strict.\\

When $K_X+B$ is big$/Z$, the \emph{lc model} of $(X,B)$ over $Z$ is a weak lc model $(Y,B_Y)$
over $Z$ with $K_Y+B_Y$ ample$/Z$.

On the other hand, a log birational model $(Y,B_Y)$ of $(X,B)$ is called a 
\emph{Mori fibre space} of $(X,B)$ over $Z$ if there is a $K_Y+B_Y$-negative extremal projective 
contraction $Y\to T/Z$, and if for any prime divisor $D$ on birational models of $X$ we have
$$
a(D,X,B)\le a(D,Y,B_Y)
$$
with strict inequality if $D\subset X$ and if it is exceptional/$Y$,

Note that the above definitions are slightly different from the traditional definitions. However, 
if $(X,B)$ is plt (hence also klt) the definitions coincide.

Let $(X,B)$ be an lc pair over $k$. A $\Q$-factorial dlt pair $(Y,B_Y)$ is 
a \emph{$\Q$-factorial dlt model} of $(X,B)$ if there is a projective birational 
morphism $f\colon Y\to X$ such that $K_Y+B_Y=f^*(K_X+B)$ and such that every exceptional prime
divisor of $f$ has coefficient $1$ in $B_Y$. On the other hand, 
when $(X,B)$ is klt, a pair $(Y,B_Y)$ with terminal singularities 
is a \emph{terminal model} of $(X,B)$ if there is a projective birational 
morphism $f\colon Y\to X$ such that $K_Y+B_Y=f^*(K_X+B)$.

\subsection{Keel's results}\label{ss-Keel} 
We recall some of the results of Keel which will be used in this paper. 
For a nef $\Q$-Cartier divisor $L$ on a projective scheme $X$ over $k$, the 
\emph{exceptional locus} $\mathbb{E}(L)$ is the union of those 
positive-dimensional integral subschemes $Y\subseteq X$ such that $L|_Y$ is not big, i.e. $(L|_Y)^{\dim Y}=0$. 
By [\ref{CMM}], $\mathbb{E}(L)$ coincides with the augmented base locus ${\bf{B_+}}(L)$.
We say $L$ is \emph{endowed with a map} $f\colon X\to V$, where $V$ is an algebraic space over $k$, 
if: an integral subscheme $Y$ is contracted by $f$ (i.e. $\dim Y>\dim f(Y)$) 
if and only if $L|_Y$ is not big.

\begin{thm}[{[\ref{Keel}, 1.9]}]\label{t-Keel-1}
Let $X$ be a projective scheme over $k$ and $L$ a nef $\Q$-Cartier divisor on $X$. 
Then 

$\bullet$ $L$ is semi-ample if and only if $L|_{\mathbb{E}(L)}$ is semi-ample;

$\bullet$ $L$ is endowed with a map if and only if $L|_{\mathbb{E}(L)}$ is endowed with a map.
\end{thm}

The theorem does not hold if $k$ is of characteristic $0$. When $L|_{\mathbb{E}(L)}\equiv 0$, then 
$L|_{\mathbb{E}(L)}$ is automatically endowed with the constant map $\mathbb{E}(L)\to \rm{pt}$ hence 
$L$ is endowed with a map. This is particularly useful for studying $3$-folds because it is often 
not difficult to show that $L|_{\mathbb{E}(L)}$ is endowed with a map, eg  when $\dim \mathbb{E}(L)=1$.

\begin{thm}[{[\ref{Keel}, 0.5]}]\label{t-Keel-2}
Let $(X,B)$ be a projective $\Q$-factorial pair of dimension $3$ over $k$ with $B$ a $\Q$-divisor. 
Assume that $A$ is an ample $\Q$-divisor such that 
$L=K_X+B+A$ is nef and big. Then $L$ is endowed with a map.
\end{thm}

In particular, when $L^\perp$ is an extremal ray, 
then we can contract $R$ to an algebraic space by the map associated to $L$.
Thus such an extremal ray is generated by the class of some curve.

We also recall the following cone theorems which we will use repeatedly in Section \ref{s-ext-rays}. 
Note that these theorems (as well as \ref{t-Keel-2}) do not assume singularities to be lc.

\begin{thm}[{[\ref{Keel}, 0.6]}]\label{t-Keel-3}
Let $(X,B)$ be a projective $\Q$-factorial pair of dimension $3$ over $k$ with $B$ a $\Q$-divisor. 
Assume that $K_X+B\sim_\Q M$ for some $M\ge 0$. Then there is a countable 
number of curves $\Gamma_i$ such that 

$\bullet$ $\overline{NE}(X)=\overline{NE}(X)_{K_X+B\ge 0}+\sum_i \R [\Gamma_i]$,

$\bullet$  all but finitely many of the $\Gamma_i$ are rational curves satisfying 
$-3\le (K_X+B)\cdot \Gamma_i<0$, and  

$\bullet$ the rays $\R [\Gamma_i]$ do not accumulate inside  $\overline{NE}(X)_{K_X+B<0}$.\\
\end{thm}

\begin{thm}[{[\ref{Keel}, 5.5.2]}]\label{t-Keel-4}
Let $(X,B)$ be a projective $\Q$-factorial pair of dimension $3$ over $k$. 
Assume that 
$$
L=K_X+B+H\sim_\R A+M
$$ 
is nef where $H,A$ are ample $\R$-divisors, and $M\ge 0$. 
Then any extremal ray of $L^\perp$ is generated by some curve $\Gamma$ 
such that either 

$\bullet$ $\Gamma$ is a component of the singular locus of $B+M$ union 
with the singular locus of $X$, or

$\bullet$ $\Gamma$ is a rational curve satisfying $-3\le (K_X+B)\cdot \Gamma<0$.\\
\end{thm}

\begin{rem}\label{ss-good-exc-locus}
Let $(X,B)$ a projective lc pair of dimension $3$ over $k$ with $B$ a $\Q$-boundary, 
and $H$ an ample $\Q$-divisor. Assume that $L=K_X+B+H$ is  
nef and big. Moreover, suppose that each connected component of 
 $\mathbb{E}(L)$ is inside some normal irreducible component $S$ of $\rddown{B}$. 
Then  $L|_S$ is semi-ample for such components (cf. [\ref{Tanaka}]) hence $L|_{\mathbb{E}(L)}$ is semi-ample 
and this in turn implies that $L$ is semi-ample by Theorem \ref{t-Keel-1}. 
\end{rem}

\section{Extremal rays and special kinds of LMMP}\label{s-ext-rays}

As usual the varieties and algebraic spaces in this section are defined over $k$ of char $p>0$.

\subsection{Extremal curve of a ray} \label{ss-ext-curve}
Let $X$ be a projective variety and $H$ a fixed ample Cartier divisor. 
Let $R$ be a ray of  $\overline{NE}(X)$ which is generated by some curve $\Gamma$. 
Assume that 
$$
H\cdot \Gamma=\min\{H\cdot C\mid \mbox{$C$ generates $R$} \}
$$
 In this case, we say $\Gamma$ 
is an \emph{extremal curve} of $R$ (in practice we do not mention $H$ and assume that 
it is already fixed). Let $C$ be any other curve 
generating $R$.
Assume that $D\cdot R<0$ for some $\R$-Cartier divisor $D$. Since $\Gamma$ and $C$ both generate $R$, 
$$
\frac{D\cdot C}{H\cdot C}=\frac{D\cdot \Gamma}{H\cdot \Gamma}
$$
hence 
$$
D\cdot \Gamma={D\cdot C}(\frac{H\cdot \Gamma}{H\cdot C})\ge D\cdot C
$$
which implies that 
$$
D\cdot \Gamma=\max\{D\cdot C\mid \mbox{$C$ generates $R$} \}
$$

\subsection{Negative extremal rays}\label{ss-ext-rays} 
Let $(X,B)$ be a projective $\Q$-factorial pair of dimension $3$. 
 Let $R$ be a $K_X+B$-negative extremal ray. Assume that there is a 
boundary $\Delta$ such that $K_X+\Delta$ is pseudo-effective and 
$(K_X+\Delta)\cdot R<0$. By adding a small ample divisor and perturbing the 
coefficients we can assume that $\Delta$ is rational and that $K_X+\Delta$ is big. 
Then by Theorem \ref{t-Keel-3}, 
$R$ is generated by some extremal curve and $R$ is an isolated extremal ray of 
$\overline{NE}(X)$. 

Now assume that $K_X+B$ is pseudo-effective and let $A$ be an ample $\R$-divisor. 
Then for any $\epsilon>0$, there are only finitely many $K_X+B+\epsilon A$-negative 
extremal rays: assume that this is not the case; then we can find a $\Q$-boundary $\Delta$ 
such that $K_X+\Delta$ is big and 
$$
K_X+B+\epsilon A\sim_\R K_X+\Delta+G
$$ where 
$G$ is ample; so there are also infinitely many $K_X+\Delta$-negative 
extremal rays; but $K_X+\Delta$ is 
big hence by Theorem \ref{t-Keel-3} all but finitely many of the 
$K_X+\Delta$-negative extremal 
rays are generated by extremal curves $\Gamma$ with 
$-3\le (K_X+\Delta)\cdot \Gamma<0$; if $(K_X+B+\epsilon A)\cdot \Gamma<0$,  
then $G\cdot \Gamma\le 3$; since $G$ is ample, there can be only 
finitely many such $\Gamma$ up to numerical equivalence.

Let $R$ be a $K_X+B$-negative extremal ray where $K_X+B$ is not necessarily pseudo-effective. 
But assume that there is a pseudo-effective $K_X+\Delta$ with  
$(K_X+\Delta)\cdot R<0$. By the remarks above we may assume $\Delta$ is rational, $K_X+\Delta$ big, 
and that there are only finitely many $K_X+\Delta$-negative extremal rays.
Therefore, we can find an ample $\Q$-divisor 
$H$ such that $L=K_X+\Delta+H$ is nef and big and $L^\perp=R$. That is, $L$ is a \emph{supporting divisor} 
of $R$. Moreover,  
$R$ can be contracted to an algebraic space, by Theorem \ref{t-Keel-2}. 
More precisely, there is a 
contraction $X\to V$ to an algebraic space such that  
it contracts a curve $C$ if and only if $L\cdot C=0$ if and only if the class $[C]\in R$.

\subsection{More on negative extremal rays}\label{ss-ext-rays-II} 
Let $(X,B)$ be a projective $\Q$-factorial pair of dimension $3$. 
Let $\mathcal{C}\subset \overline{NE}(X)$ be one of the following:

$(1)$ $\mathcal{C}=\overline{NE}(X/Z)$ for a given projective contraction $X\to Z$ such that 
$K_X+B\equiv P+M/Z$ where $P$ is nef$/Z$ and $M\ge 0$ (this is a \emph{weak Zariski decomposition}; 
see \ref{ss-WZD}); or

$(2)$ $\mathcal{C}=N^\perp$ for some nef and big $\Q$-divisor $N$;

We will show that in both cases, each $K_X+B$-negative extremal ray 
$R$ of $\mathcal{C}$ is generated by an extremal curve $\Gamma$, and for all but finitely many of 
those rays we have $-3\le (K_X+B)\cdot \Gamma<0$. 

We first deal with case (1). Fix a 
$K_X+B$-negative extremal ray $R$ of $\mathcal{C}$. By replacing $P$ we can 
assume that $K_X+B=P+M$. Let $A$ be an ample $\R$-divisor and 
$T$ be the pullback of a sufficiently ample divisor on $Z$ 
so that $K_X+B+A+T$ is big and $(K_X+B+A+T)\cdot R<0$. 
By \ref{ss-ext-rays}, there is a nef and big $\Q$-divisor $L$ with $L^\perp=R$. Moreover, 
we may assume that if   
$l\gg 0$, then 
$$
Q_1:=K_X+B+T+lL+A
$$
 is nef and big and $Q_1^\perp=R$. 
By construction, $T+lL+A$ is ample, $P+T+lL+A$ is also ample, and 
$$
K_X+B+T+lL+A=P+T+lL+A+M
$$ 
Therefore, by Theorem \ref{t-Keel-4},  $R$ is generated by some curve 
$\Gamma$ satisfying $-3\le (K_X+B)\cdot \Gamma<0$ or $R$ is generated by some curve 
in the singular locus of $B+M$ or $X$. There are only finitely many possibilities in the latter 
case. The claim then follows.

Now we deal with case (2). Fix a $K_X+B$-negative extremal ray $R$ of $\mathcal{C}$. 
Since $N$ is nef and big, for some $n>0$, 
$$
K_X+B+nN\sim_\R G+S
$$ 
where $G$ is ample and $S\ge 0$. 
By \ref{ss-ext-rays}, there is a nef and big $\Q$-divisor $L$ with $L^\perp=R$. 
Moreover, for some  
$l\gg 0$ and some ample $\R$-divisor $A$, 
$$
Q_2:=K_X+B+nN+lL+A
$$ 
is nef and big with $Q_2^\perp=R$. 
Now, $nN+lL+A$ is ample, $G+lL+A$ is ample, and 
$$
K_X+B+nN+lL+A\sim_\R G+lL+A+S
$$
Therefore, by Theorem \ref{t-Keel-4},  $R$ is generated by some curve 
$\Gamma$ satisfying $-3\le (K_X+B)\cdot \Gamma<0$ or $R$ is generated by some curve 
in the singular locus of $B+S$ or $X$. There are only finitely many possibilities in the latter 
case. The claim then follows.

Assume that $R$ is a $K_X+B$-negative extremal ray of $\mathcal{C}$, in either case. 
Then the above arguments show that there is a $\Q$-boundary $\Delta$ and an ample $\Q$-divisor 
$H$ such that $K_X+\Delta$ is big, $(K_X+\Delta)\cdot R<0$, and $L=K_X+\Delta+H$ is nef and 
big with $L^\perp=R$. Therefore, as in \ref{ss-ext-rays}, $R$ can be contracted via a contraction $X\to V$ to an 
algebraic space. Moreover, if $B$ is rational, then we can find an ample $\Q$-divisor 
$H'$ such that $L'=K_X+B+H'$ is nef and big and again $L'^\perp=R$.

\subsection{Extremal rays given by scaling.}\label{ss-ext-rays-scaling} 
Let $(X,B)$ be a projective $\Q$-factorial 
 pair of dimension $3$. Assume that either $\mathcal{C}=\overline{NE}(X/Z)$ for some 
projective contraction $X\to Z$ such that $K_X+B\equiv M/Z$ for some $M\ge 0$, or $\mathcal{C}=N^\perp$ for some 
nef and big $\Q$-divisor $N$. 
 In addition assume that $(X,B+C)$ is a pair for some $C\ge 0$ and that $K_X+B+C$ is nef on 
$\mathcal{C}$, that is, $(K_X+B+C)\cdot R\ge 0$ for every extremal ray $R$ of $\mathcal{C}$. 
Let 
$$
\lambda=\inf\{t\ge 0\mid K_X+B+tC \mbox{~is nef on $\mathcal{C}$}\}
$$
Then we will see that either $\lambda=0$ or there is an extremal ray $R$ of $\mathcal{C}$ 
such that $(K_X+B+\lambda C)\cdot R=0$ and $(K_X+B)\cdot R<0$. Assume $\lambda>0$. 
If the claim is not true, then there exist a sequence of numbers $t_1<t_2<\cdots$ approaching 
$\lambda$ and extremal rays $R_i$ of $\mathcal{C}$ such that 
$(K_X+B+t_i C)\cdot R_i=0$ and $(K_X+B)\cdot R_i<0$. 

First assume that $\mathcal{C}=N^\perp$ for some nef and big $\Q$-divisor $N$.
We can write a finite sum $K_X+B=\sum_j r_j(K_X+B_j)$ where $r_j\in (0,1]$, $\sum r_j=1$, and 
$(X,B_j)$ are pairs with $B_j$ being rational. By \ref{ss-ext-rays-II}, we may assume 
that each $R_i$ is generated by some extremal curve $\Gamma_i$ with $-3\le (K_X+B_j)\cdot \Gamma_i$ 
for each $j$. This implies that there are only finitely many possibilities 
for the numbers $(K_X+B)\cdot \Gamma_i$. A similar reasoning shows that there are only finitely many 
possibilities for the numbers $(K_X+B+\frac{\lambda}{2}C)\cdot \Gamma_i$ hence there are also only finitely many 
possibilities for the numbers $C\cdot \Gamma_i$. But then this implies that there are finitely 
many $t_i$, a contradiction.

Now assume that $\mathcal{C}=\overline{NE}(X/Z)$ for some projective contraction $X\to Z$ 
such that $K_X+B\equiv M/Z$ for some $M\ge 0$. 
Then we can write $K_X+B=\sum_j r_j(K_X+B_j)$ and  $M=\sum_j r_jM_j$ 
where $r_j\in (0,1]$, $\sum r_j=1$, $(X,B_j)$ are pairs with $B_j$ being rational, 
$K_X+B_j\equiv M_j/Z$, and $M_j\ge 0$. 
To find such a decomposition we argue as in [\ref{BP}, pages 96-97].
Let $V$ and $W$ be the $\R$-vector spaces generated by the 
components of $B$ and $M$ respectively. For a vector $v\in V$ (resp. $w\in W$) we denote the 
corresponding $\R$-divisor by 
$B_v$ (resp. $M_w$). Let $F$ be the set of those $(v,w)\in V\times W$ such that 
$(X,B_v)$ is a pair, $M_w\ge 0$, and $K_X+B_v\equiv M_w/Z$. 
Then $F$ is defined by a finite 
number of linear equalities and 
inequalities with rational coefficients. If $B=B_{v_0}$ and $M=M_{w_0}$ are the given divisors, then $(v_0,w_0)\in F$ 
hence it belongs to some polytope in $F$ with rational vertices. The vertices of the polytope give the $B_j,M_j$. 
The rest of the proof is as in the last paragraph.

\subsection{LMMP with scaling}\label{ss-g-LMMP-scaling}
Let $(X,B)$ be a projective $\Q$-factorial pair of dimension $3$. 
Assume that either $\mathcal{C}=\overline{NE}(X/Z)$ for some 
projective contraction $X\to Z$ such that $K_X+B\equiv M/Z$ for some $M\ge 0$, or $\mathcal{C}=N^\perp$ for some 
nef and big $\Q$-divisor $N$. 
 In addition assume that $(X,B+C)$ is a pair for some $C\ge 0$ and that $K_X+B+C$ is nef on 
$\mathcal{C}$. 

If $K_X+B$ is not nef on $\mathcal{C}$, by \ref{ss-ext-rays-scaling}, there is an extremal ray $R$ 
of $\mathcal{C}$ such that $(K_X+B+\lambda C)\cdot R=0$ and $(K_X+B)\cdot R<0$ 
where $\lambda$ is the smallest  number such that $K_X+B+\lambda C$ is nef on 
$\mathcal{C}$. Assume that $R$ can be contracted by a projective morphism. The contraction is birational 
because $L\cdot R=0$ for some nef and big $\Q$-Cartier divisor $L$ (see \ref{ss-ext-rays-II}). 
Assume that $X\bir X'$ is the corresponding divisorial contraction or flip, and assume that 
$X'$ is $\Q$-factorial. 
Let $\mathcal{C}'$ be the cone given by $\mathcal{C}'=\overline{NE}(X'/Z)$ or $\mathcal{C}'=(N')^\perp$ 
 corresponding to the above cases. 
Let $\lambda'$ be the smallest nonnegative number such that $K_{X'}+B'+\lambda' C'$ is nef on 
$\mathcal{C}'$. If $\lambda'>0$, then there is an extremal ray $R'$ 
of $\mathcal{C}'$ such that $(K_{X'}+B'+\lambda' C')\cdot R'=0$ and $(K_{X'}+B')\cdot R'<0$. 
Assume that $R'$ can be contracted and so on. Assuming that all the necessary ingredients 
exist, the process gives a special kind of 
LMMP which we may refer to as \emph{LMMP$/\mathcal{C}$ on $K_X+B$ with scaling of $C$}.  
Note that $\lambda\ge \lambda'\ge \cdots$

If $\mathcal{C}=\overline{NE}(X/Z)$, we also refer to the above LMMP as the LMMP$/Z$ 
on $K_X+B$ with scaling of $C$. If $\mathcal{C}=N^\perp$, and if 
$N$ is endowed with a map $X\to V$ to an algebraic space, we refer to the 
above LMMP as the LMMP$/V$ on $K_X+B$ with scaling of $C$. 

In practice, when we run an LMMP with scaling, $(X,B)$ is $\Q$-factorial dlt 
and each extremal ray in the process intersects some component of $\rddown{B}$ 
negatively. In particular, such rays can be contracted by projective morphisms and 
the $\Q$-factorial property is preserved by the LMMP (see \ref{ss-pl-ext-rays}). 
If the required 
flips exist then the LMMP terminates by special termination (see \ref{ss-special-termination}).

\subsection{Extremal rays given by a weak Zariski decomposition}\label{ss-ext-ray-WZD}
Let $(X,B)$ be a projective $\Q$-factorial pair of dimension $3$ and $X\to Z$ a projective contraction such that 
\begin{enumerate}
\item $K_X+B\equiv P+M/Z$, $P$ is nef$/Z$, $M\ge 0$, and 
\item $\Supp M\subseteq \rddown{B}$. 
\end{enumerate}

Let
$$
\mu=\sup \{t\in [0,1] \mid P+tM ~~~\mbox{is nef$/Z$}\}. 
$$
Assume that $\mu<1$. We will show that there is an extremal ray $R/Z$ 
such that $(K_X+B)\cdot R<0$ and $(P+\mu M)\cdot R=0$. 

 Replacing $P$ with $P+\mu M$ we may assume that $\mu=0$. 
Then by definition of $\mu$, $P+\epsilon'M$ is not nef$/Z$ for any $\epsilon'>0$. In particular, for any 
$\epsilon'>0$ there is a $K_X+B$-negative extremal ray $R/Z$ 
such that $(P+\epsilon'M)\cdot R<0$ but $(P+\epsilon M)\cdot R=0$ for some $\epsilon\in [0,\epsilon')$.
If  there is no $K_X+B$-negative extremal ray $R/Z$ such that $P\cdot R=0$, then there is an infinite  
strictly decreasing 
sequence of sufficiently small positive real numbers $\epsilon_i$ and $K_X+B$-negative extremal rays $R_i/Z$ 
such that $\lim_{i\to \infty} \epsilon_i=0$ and
$(P+\epsilon_i M)\cdot R_i=0$. 

We may assume that for each $i$, there is an extremal curve 
$\Gamma_i$ generating $R_i$ such that $-3\le (K_X+B)\cdot \Gamma_i<0$ (see \ref{ss-ext-rays-II}).
Since $\Supp M\subseteq \rddown{B}$, 
there is a small $\delta>0$ such that  $(K_X+B-\delta M)\cdot \Gamma_i<0$ 
for each $i$, $B-\delta M\ge 0$, 
and $\Supp (B-\delta M)=\Supp B$. 
We have
$$
K_X+B-\delta M\equiv P+(1-\delta)M/Z
$$  

By replacing the sequence of extremal rays with a subsequence, 
we can assume that each component $S$ of $M$ satisfies: either $S\cdot R_i\ge 0$ for every $i$, or 
$S\cdot R_i<0$ for every $i$. Pick a component $S$. If $S\cdot R_i\ge 0$ for each $i$, 
then by \ref{ss-ext-rays-II}, we may assume that 
$$
-3\le (K_X+B-\delta M)\cdot \Gamma_i<0
$$ 
and
$$
-3\le (K_X+B-\delta M-\tau S)\cdot \Gamma_i<0
$$ 
for every $i$ where $\tau>0$ is a small number.
In particular, this means that $S\cdot \Gamma_i$ 
is bounded from below and above. On the other hand, if $S\cdot R_i< 0$ for each $i$, then by considering 
$K_X+B-\delta M+\tau S$ and arguing similarly we can show that again $S\cdot \Gamma_i$ 
is bounded from below and above. In particular, there are only finitely many 
possibilities for the numbers $M\cdot \Gamma_i$.
Therefore, 
$$
\lim_{i\to \infty} P\cdot \Gamma_i=\lim_{i\to \infty} -\epsilon_iM\cdot \Gamma_i=0
$$

Write $K_X+B=\sum_j r_j(K_X+B_j)$ where $r_j\in (0,1]$, $\sum r_j=1$, and 
$(X,B_j)$ are pairs with $B_j$ being rational. 
We can assume that 
each component of $B-B_j$ has irrational coefficient in $B$ hence 
 $B-B_j$ and $M$ have no common components because $\Supp M\subseteq \rddown{B}$. 
Assume $(K_X+B_j)\cdot \Gamma_i<0$ for some $i,j$. Let $S$ be a component of $M$ such that
 $S\cdot \Gamma_i<0$, and let $S^\nu$ be its normalization. 
 Let $K_{S^\nu}+B_{j,S^\nu}=(K_X+B_j)|_{S^\nu}$ (see Section \ref{s-adjunction} for adjunction 
formulas of this type). On the other hand, by  \ref{ss-ext-rays-II}, 
there is an ample $\Q$-divisor $H$ such that $Q=K_X+B_j+H$ is nef and big and $R_i=Q^\perp$.
Now the face $(Q|_{S^\nu})^\perp$  of $\overline{NE}(S^\nu/Z)$  is generated by finitely many curves 
$\Lambda_1^\nu, \dots, \Lambda_r^\nu$ such that $\alpha_j\le (K_{S^\nu}+B_{j,S^\nu})\cdot \Lambda_l^\nu<0$ 
where $\alpha_j$ depends on $({S^\nu},B_{j,S^\nu})$ but does not depend on $i$,   
by Tanaka [\ref{Tanaka}, Theorem 4.4, Remark 4.5]. Let $\Lambda_l$ be the image of $\Lambda_l^\nu$ 
under the map $S^\nu\to X$.
 Since $R_i=Q^\perp$ and $Q\cdot \Lambda_l=0$, each $\Lambda_l$ also generates $R_i$. But as $\Gamma_i$ is 
extremal, perhaps after replacing the $\alpha_j$, we get 
$$
\alpha_j\le (K_X+B_j)\cdot \Lambda_l\le (K_X+B_j)\cdot \Gamma_i<0
$$
by \ref{ss-ext-curve}. 

On the other hand, 
since 
$$
-3\le (K_X+B)\cdot \Gamma_i=\sum_j r_j(K_X+B_j)\cdot \Gamma_i<0
$$ 
for each $i$, we deduce that  $(K_X+B_j)\cdot \Gamma_i$ is  
bounded from below and above for each $i,j$ which in turn implies that there are only finitely many 
possibilities for $(K_X+B)\cdot \Gamma_i$. Recalling that there are also finitely many 
possibilities for $M\cdot \Gamma_i$, we get a contradiction as 
$$
0<P\cdot \Gamma_i=(K_X+B)\cdot \Gamma_i-M\cdot \Gamma_i
$$ 
but $\lim_{i\to \infty} P\cdot \Gamma_i=0$.

\subsection{LMMP using a weak Zariski decomposition}\label{ss-LMMP-WZD}
Let $(X,B)$ be a projective $\Q$-factorial pair of dimension $3$ and $X\to Z$ a projective contraction such that 
$K_X+B\equiv P+M/Z$ where $P$ is nef$/Z$, $M\ge 0$, and 
$\Supp M\subseteq \rddown{B}$. Let $\mu$ be the largest number such that $P+\mu M$ is nef$/Z$. 
Assume $\mu<1$. Then, by \ref{ss-ext-ray-WZD}, there is an extremal ray $R/Z$ such that 
$(K_X+B)\cdot R<0$ and $(P+\mu M)\cdot R=0$. By replacing $P$ with $P+\mu M$ we may assume that 
$P\cdot R=0$. Assume that $R$ can be contracted by a projective morphism and that 
it gives a divisorial contraction or a log flip 
$X\bir X'/Z$ with $X'$ being $\Q$-factorial. Obviously,  
$K_{X'}+B'\equiv P'+M'/Z$ 
where $P'$ is nef$/Z$, $M'\ge 0$, and $\Supp M'\subseteq \rddown{B'}$. 
Continuing this process we obtain a particular kind of LMMP 
which we will refer to as the  \emph{LMMP using a weak Zariski decomposition} or more specifically 
the \emph{LMMP$/Z$ on $K_X+B$ using $P+M$}. When we need this LMMP below we will make sure that 
all the necessary ingredients exist.\\

\section{Adjunction}\label{s-adjunction}

The varieties in this section are over $k$ of arbitrary characteristic.
 We will use some of the 
results of Koll\'ar [\ref{Kollar-sing}] to prove an adjunction formula.
Let $\Lambda$ be a DCC set of numbers in $[0,1]$. 
Then the hyperstandard set 
$$
\mathfrak{S}_\Lambda=\{\frac{m-1}{m}+\sum \frac{l_ib_i}{m}\le 1 \mid m\in \N\cup\{\infty\}, l_i\in \Z^{\ge 0}, b_i\in \Lambda\}
$$
also satisfies DCC.

Now let $(X,B)$ be a pair and $S$ a component of 
$\rddown{B}$. Let $S^\nu \to S$ be the normalization. 
Following a suggestion of Koll\'ar, 
we will show that the pullback of $K_X+B$ to $S^{\nu}$ can be canonically written 
as $K_{S^\nu}+B_{S^\nu}$ for some  $B_{S^\nu}\ge 0$ which is called the \emph{different}. 
Moreover, if $(X,B)$ is lc outside a codimension $3$ closed subset and if 
the coefficients of $B$ belong to $\Lambda$, then we show $B_{S^\nu}$ is a boundary with coefficients in 
$\mathfrak{S}_\Lambda$. When there is a 
log resolution $f\colon W\to X$, it is easy to define $B_{S^\nu}$: 
let $K_W+B_W=f^*(K_X+B)$ and let 
$K_T+B_T=(K_W+B_W)|_T$ where $T$ is the birational transform of $S$.  
Next, let $B_{S^\nu}$ be the pushdown of $B_T$ via $T\to S^\nu$.
However, since existence of log resolutions is not known in general, we follow a different path, that is, 
that of [\ref{Kollar-sing}, Section 4.1]. 
Actually, in this paper we will need this construction 
only when $\dim X\le 3$ in which case log resolutions exist. 
 
The characteristic $0$ case of the results mentioned is due to Shokurov [\ref{Shokurov}, Corollary 3.10]. 
His idea is to cut by appropriate hyperplane sections and reduce the problem to the case when $X$ is a surface.
If the index of $K_X+S$ is $1$ one proves the claim by direct calculations on a 
resolution. If the index is more than $1$ one then uses the index $1$ cover. Unfortunately this 
does not work in positive characteristic.

\begin{prop}\label{p-adjunction-existence} 
Let $(X,B)$ be a pair, $S$ be a component of $\rddown{B}$, and $S^\nu\to S$ be the normalization.
Then there is a canonically determined $\R$-divisor $B_{S^\nu}\ge 0$  
such that 
$$
K_{S^\nu}+B_{S^\nu}\sim_\R (K_X+B)|_{S^\nu}
$$
where $|_{S^\nu}$ means pullback to ${S^\nu}$ by the induced morphism $S^\nu\to X$.
\end{prop}
\begin{proof}
If $K_X+B$ is $\Q$-Cartier, then the statement is proved in 
[\ref{Kollar-sing}, 4.2 and 4.5]. In fact, [\ref{Kollar-sing}, 4.2] defines $\Delta_{S^\nu}$ 
in general when $\Delta$ is a $\Q$-divisor with arbitrary rational coefficients,  
$S$ is a component of $\Delta$ with coefficient $1$, and $K_X+\Delta$ is $\Q$-Cartier (but $\Delta_{S^\nu}$ is not effective in general).

Let $U$ be the 
$\R$-vector space generated by the components of $B$.  There is a    
rational affine subspace $V$ of $U$ containing $B$ and with minimal dimension. 
Since $V$ has minimal dimension, $\Delta-B$ is supported in the irrational part of 
$B$ for every $\Delta\in V$. Thus the coefficient of $S$ in $\Delta$ is $1$ 
for every $\Delta\in V$.

Let $V_\Q$ be the underlying $\Q$-affine space of $V$.  Let 
$$
W_\Q=\{\Delta_{S^\nu} \mid \Delta\in V_\Q\}
$$
If $\Delta=\sum r_j\Delta^j$ where $r_j> 0$ is rational, $\sum r_j=1$, and $\Delta^j\in V_\Q$, then 
the construction of [\ref{Kollar-sing}, 4.2] shows that $\Delta_{S^\nu}=\sum r_j\Delta^j_{S^\nu}$. 
Therefore, $W_\Q$ is a $\Q$-affine space and the map $\alpha \colon V_\Q\to W_\Q$ 
sending $\Delta$ to $\Delta_{S^\nu}$ is an affine map. Letting $W$ be the $\R$-affine space generated by $W_\Q$, 
we get an induced affine map $V\to W$ which sends $B$ to some element $B_{S^\nu}$. 
Writing $B=\sum r_j\Delta^j$ where $r_j> 0$, $\sum r_j=1$, and $0\le \Delta^j\in V_\Q$, we see that  
$B_{S^\nu}=\sum r_j\Delta^j_{S^\nu}\ge 0$. Moreover, by construction 
$$
K_{S^\nu}+B_{S^\nu}=\sum r_j(K_{S^\nu}+\Delta^j_{S^\nu})\sim_\R \sum r_j(K_X+\Delta^j)|_{S^\nu} =(K_X+B)|_{S^\nu}
$$\\   
\end{proof}

Note that in general $B_{S^\nu}$ is not a boundary, i.e. its coefficients may not 
be in $[0,1]$. 

\begin{prop}\label{p-adjunction-DCC}
Let $\Lambda\subseteq [0,1]$ be a DCC set of real numbers. 
Let $(X,B)$ be a pair, $S$ be a component of $\rddown{B}$, $S^\nu\to S$ be the normalization, 
and $B_{S^\nu}$ be the divisor given by Proposition \ref{p-adjunction-existence}.   
Assume that 

$\bullet$ $(X,B)$ is lc outside a codimension $3$ closed subset, and 
 
$\bullet$ the coefficients of $B$ are in $\Lambda$.\\

Then $B_{S^\nu}$ is a boundary with coefficients in $\mathfrak{S}_\Lambda$. 
More precisely: write $B=S+\sum_{i\ge 2} b_iB_i$, 
let $V^\nu$ be a prime divisor on $S^\nu$ and let $V$ be its image on $S$;
then there exists $m\in\N\cup \{\infty\}$ depending only on $X,S$ and $V$, and there exist
nonnegative integers $l_i$ depending only on $X,S,B_i$ and $V$, such that 
the coefficient of $V^\nu$ in $B_{S^\nu}$ is equal to 
$$
\frac{m-1}{m}+\sum_{i\ge 2} \frac{l_ib_i}{m}
$$
\end{prop}

\begin{lem}\label{l-lc-codim2}
Let $(X,B)$ be a pair which is lc outside a codimension $3$ closed subset. 
Then we can write $B=\sum r_jB^j$ where $r_j> 0$, $\sum r_j=1$, $B^j$ are $\Q$-boundaries, 
and $(X,B^j)$ are lc outside a codimension $3$ closed subset. 
\end{lem}
\begin{proof}
As in the proof of Proposition \ref{p-adjunction-existence}, there is a  
rational affine space $V$ of divisors, containing $B$, such that  
$K_X+\Delta$ is $\R$-Cartier for every $\Delta\in V$. The set of those $\Delta\in V$ 
with coefficients in $[0,1]$ is a rational polytope $\mathcal{P}$ containing $B$.   
We want to show that there is a rational polytope $\mathcal{L}\subseteq \mathcal{P}$, 
containing $B$, such that $(X,\Delta)$ is lc outside a fixed codimension $3$ closed subset, 
for every $\Delta\in  \mathcal{L}$.
If $(X,B)$ has a log resolution, then existence of $\mathcal{L}$ can be proved using 
the same arguments as in [\ref{Shokurov}, 1.3.2]. 

The pair $(X,B)$ is log smooth outside some codimension $2$ closed subset $Y$. 
In particular, $(X,\Delta)$ is lc outside $Y$, for every $\Delta\in  \mathcal{P}$. 
Shrinking $X$ we can assume $Y$ is of pure codimension $2$ and that $(X,B)$ is lc everywhere. 
Assume that  
for each component $R$ of $Y$, there is a rational polytope $\mathcal{L}_R\subseteq \mathcal{P}$, 
containing $B$, such that $(X,\Delta)$ is lc near the generic point of $R$, 
for every $\Delta\in  \mathcal{L}_R$. Then we can take $\mathcal{L}$ to be any rational 
polytope, containing $B$, inside the intersection of the  $\mathcal{L}_R$. 

Existence of $\mathcal{L}_R$ is a local problem near the generic point of $R$. 
By replacing $X$ with $\Spec \mathcal{O}_{X,R}$ we are reduced to the situation in which 
$X$ is a normal excellent scheme of dimension $2$ (see 
[\ref{Kollar-sing}, 3.3] for notion of lc pairs in this setting). 
Now $(X,B)$ has a log resolution (cf. see [\ref{Shafarevich}, page 28 and following remarks, and page 72]).  
So existence of $\mathcal{L}_R$ can be proved again as in [\ref{Shokurov}, 1.3.2].\\ 
\end{proof}

\begin{proof}(of Proposition \ref{p-adjunction-DCC})
Assume that the proposition holds whenever $K_X+B$ is $\Q$-Cartier. 
In the general case, that is, when $K_X+B$ is only $\R$-Cartier,  
we can use Lemma \ref{l-lc-codim2} to write 
$B=\sum r_jB^j$ where $r_j> 0$, $\sum r_j=1$, $B^j$ are $\Q$-boundaries, and $(X,B^j)$ are lc 
outside a codimension $3$ closed subset. Moreover, we can assume $S$ is a component of 
$\rddown{B^j}$ for each $j$ since we can choose the $B^j$ so that $B-B^j$ are supported on the 
irrational part of $B$. 
Then $B_{S^\nu}=\sum r_jB^j_{S^\nu}$ (see the proof of Proposition \ref{p-adjunction-existence}). 
Write $B^j=S+\sum_{i\ge 2} b^j_iB_i$.  
By assumption, there exists $m\in\N\cup \{\infty\}$ depending only on $X,S$ and $V$, and 
there exist nonnegative integers $l_i$ depending only on $X,S,B_i$ and $V$, such that 
the coefficient of $V^\nu$ in $B^j_{S^\nu}$ is equal to  
$$
\frac{m-1}{m}+\sum \frac{l_ib^j_i}{m}
$$
Therefore, the coefficient of $V^\nu$ in $B_{S^\nu}$ is equal to  
$$
\frac{m-1}{m}+\sum_j r_j (\sum_i \frac{l_ib^j_i}{m})=\frac{m-1}{m}+\sum_i l_i (\sum_j \frac{r_jb^j_i}{m})=\frac{m-1}{m}+\sum_i \frac{l_ib_i}{m}
$$

So from now on we can assume that $K_X+B$ is $\Q$-Cartier.  Determining the coefficient of 
$V^\nu$ in $B_{S^\nu}$ is a local problem near the generic point of $V$. 
As in the proof of Lemma \ref{l-lc-codim2}, 
we can replace $X$ with $\Spec \mathcal{O}_{X,V}$ hence 
assume that $X$ is a normal excellent scheme of 
dimension $2$, $S$ is one-dimensional, and $V$ is a closed point. Now $(X,B)$ is lc and 
the fact that $B_{S^\nu}$ is a boundary is proved in [\ref{Kollar-sing}, 4.5].

Assume that $X$ is regular at $V$. If $S$ is not regular at $V$, then $B=S$ 
and the coefficient of $V^\nu$ in $B_{S^\nu}$ is equal to $1$ (by [\ref{Kollar-sing}, 3.45] or 
by blowing up $V$ and working on the blow up). 
But if $S$ is regular at $V$, then $S^\nu \to S$ is an isomorphism, $(K_X+S)|_{S^\nu}=K_{S^\nu}$, 
$m=1$, and $B_{S^\nu}=B|_{S^\nu}$. From these we can get the formula for the coefficient of $V^\nu$ 
as claimed. 
Thus we can assume $X$ is not regular at $V$.

Since $(X,B)$ is lc, $(X,S)$ is numerically lc (see [\ref{Kollar-sing}, 3.3] for definition of numerical lc 
which is the same as lc except that $K_X+S$ may not be $\Q$-Cartier). If  
 $(X,S)$ is not numerically plt, i.e. if there is an exceptional divisor over $V$ whose 
log discrepancy with respect to $(X,S)$ is $0$, then in fact $B=S$, and 
the coefficient of $V^\nu$ in $B_{S^\nu}$ is equal to $1$ 
by [\ref{Kollar-sing}, 3.45]. Thus we can assume $(X,S)$ is numerically plt 
which in particular implies that $S$ is regular and that $S^\nu \to S$ is an isomorphism, 
by [\ref{Kollar-sing}, 3.35].

Let $f\colon Y\to X$ be a log minimal resolution of $(X,S)$ as in [\ref{Kollar-sing}, 2.25] 
and let $S^\sim$ be the birational transform of $S$. Then $S^\sim\to S$ is an isomorphism 
and the extended dual 
graph of the resolution is of the form 
$$
\xymatrix{
\bullet \ar@{-}[r] & c_1 \ar@{-}[r] & c_2 \ar@{-}[r] &\cdots \ar@{-}[r] & c_r 
}
$$ 
where $\bullet$ corresponds to $S^\sim$, $c_i=-E_i^2$, and $E_1,\dots, E_r$ are the exceptional 
curves of $f$.
Let $M=[-E_i\cdot E_j]$ be the minus of the intersection matrix of the resolution, and let 
$m=\det M$. Then by [\ref{Kollar-sing}, 3.35.1] we have 
$$
K_Y+S^\sim+\sum e_jE_j\equiv 0/X
$$ 
for certain $e_j>0$ and $e_1=\frac{m-1}{m}$. 

Let $D\neq 0$ be an effective Weil divisor on $X$ with coefficients in $\N$. Let $d_i$ be the numbers so that 
$D^\sim+\sum d_iE_i\equiv 0/X$ where $D^\sim$ is the birational transform of $D$. 
The $d_i$ satisfy the equations 
$$
(\sum d_jE_j)\cdot E_t =-D^\sim\cdot E_t 
$$
Since $M$ has integer entries and the numbers $-D^\sim\cdot E_t$ are integers, 
by Cramer's rule, we can write $d_j=\frac{n_j}{m}$ for certain 
$n_j\in \N$. Applying this to $D=B_i$, we have $B_i^\sim+\sum d_{i,j}E_j\equiv 0/X$ for certain 
$d_{i,j}=\frac{n_{i,j}}{m}$ with $n_{i,j}\in \N$. But then 
$$
K_Y+B^\sim+\sum e_j'E_j\equiv 0/X
$$
where $B^\sim$ is the birational transform of $B$ and $e_j'=e_j+\sum_{i\ge 2} \frac{n_{i,j}b_i}{m}$. 
In particular, $e_1'=\frac{m-1}{m}+\sum_{i\ge 2} \frac{l_ib_i}{m}$ where we put $l_i:=n_{i,1}$. 
Now the coefficient of $V^\nu$ in $B_{S^\nu}$ is 
simply the coefficient of the divisor $e_1'E_1|_{S^\sim}$ which is nothing but $e_1'$.\\ 
\end{proof}

\section{Special termination}\label{s-st}

All the varieties and algebraic spaces in this section are 
over $k$ of char $p>0$ unless stated otherwise.

\subsection{Reduced components of boundaries of dlt pairs}

\begin{lem}\label{l-plt-normal}
Let $(X,B)$ be an lc pair of dimension $3$  and $S$ a component of 
$\rddown{B}$. Then we have: 

$(1)$ if the coefficients of $B$ are standard, then the coefficients of $B_{S^\nu}$ are also standard;

$(2)$ if $k$ has char $p>5$ and $(X,B)$ is $\Q$-factorial dlt, then $S$ is normal. 
\end{lem}
\begin{proof}
(1) This follows from Koll\'ar [\ref{Kollar-sing}, Corollary 3.45] (see also [\ref{Kollar-sing}, 4.4]). 
  
(2) We may assume $B=S$ by discarding all the other components, in particular, $(X,B)$ is plt 
hence $({S^\nu},B_{S^\nu})$ is klt. 
By (1), $B_{S^\nu}$ has standard coefficients. By [\ref{HX}, Theorem 3.1], $({S^\nu},B_{S^\nu})$  
is actually strongly $F$-regular. Therefore, $S$ is normal by [\ref{HX}, Theorem 4.1].\\
\end{proof}

\subsection{Pl-extremal rays}\label{ss-pl-ext-rays} 
 Let $(X,B)$ be a projective $\Q$-factorial dlt pair of dimension $3$. 
A $K_X+B$-negative extremal ray $R$ is called a \emph{pl-extremal ray} if  
$S\cdot R<0$ for some component $S$ of $\rddown{B}$. This is named after  
Shokurov's pl-flips. 
 
Assume that $k$ has char $p>5$. 
Now as in \ref{ss-ext-rays-II}, assume that $\mathcal{C}=\overline{NE}(X/Z)$ for 
some projective contraction $X\to Z$ such that 
$K_X+B\equiv P+M/Z$ where $P$ is nef$/Z$ and $M\ge 0$, or $\mathcal{C}=N^\perp$ 
for some nef and big $\Q$-divisor $N$. Let $R$ be a $K_X+B$-negative pl-extremal ray 
of $\mathcal{C}$. By \ref{ss-ext-rays-II}, we can find a $\Q$-boundary $\Delta$ and 
an ample $\Q$-divisor 
$H$ such that $\rddown{\Delta}=S$, $(K_X+\Delta)\cdot R<0$ and $L=K_X+\Delta+H$ is nef and 
big with $L^\perp=R$. Let $X\to V$ be the contraction associated to $L$ which contracts 
$R$ to an algebraic space. Every curve contracted by $X\to V$ is inside $S$. 
So the exceptional locus $\mathbb{E}(L)$ of $L$ is inside $S$. Thus 
$L$ is semi-ample by \ref{ss-good-exc-locus}. 
Therefore $X\to V$ is a projective contraction. 
In other words, pl-extremal rays can be contracted by projective morphisms. 
This was proved in 
[\ref{HX}, Theorem 5.4] when $K_X+B$ is pseudo-effective. The 
extremal rays that appear below are often pl-extremal rays. 

If $X\to V$ is a divisorial contraction put $X'=V$ but if it is a flipping contraction 
assume $X\bir X'/V$ is its flip. Then it is not hard to see that in any case $X'$ is $\Q$-factorial, 
by the following argument [\ref{Xu}]: we treat the divisorial case; 
the flipping case can be proved similarly. We can assume that $B$ is a $\Q$-boundary and $\Delta=B$. 
Let $D'$ be a prime divisor on $X'$ and 
$D$ its birational transform on $X$. 
There are rational numbers $\epsilon>0$ and $\delta$ such that 
$M:=K_X+B+H+\epsilon D+\delta S$ is nef and big, $M\equiv 0/V$, $H+\epsilon D+\delta S$ is ample, 
and $\mathbb{E}(M)=\mathbb{E}(L)=S$. Since $M|_{S}$ 
is semi-ample, $M$ is semi-ample by Theorem \ref{t-Keel-1}. That is, $M$ 
is the pullback of some ample divisor $M'$ on $X'$. But then $\epsilon D'=M'-L'$ is 
$\Q$-Cartier hence $D'$ is $\Q$-Cartier.

\subsection{Special termination}\label{ss-special-termination} 
The following important result is proved just like in characteristic $0$. 
We include the proof for convenience.

\begin{prop}\label{p-st} 
Let $(X,B)$ be a projective $\Q$-factorial dlt pair of dimension $3$ over $k$ of char $p>5$. 
Assume that we are given an LMMP on $K_X+B$, say $X_i\bir X_{i+1}/Z_i$ where $X_1=X$ 
and each $X_i\bir X_{i+1}/Z_i$ is a flip, or a divisorial contraction with $X_{i+1}=Z_i$. 
Then after finitely many steps, each remaining step of the LMMP is an isomorphism near 
the lc centres of $(X,B)$.
\end{prop}
\begin{proof}
 There are only finitely many lc centres and no new one 
can be created in the process, so we may assume that the LMMP does not contract any lc centre.
In particular, we can assume that the LMMP is an isomorphism near each lc centre of 
dimension zero. 

Now let $C$ be an lc centre of dimension one. Since $(X,B)$ is dlt, $C$ is a component 
of the intersection of two components $S,S'$ of $\rddown{B}$. Let $C_i,S_i\subset X_i$ be the birational 
transforms of $C,S$. Applying Lemma \ref{l-plt-normal}, we can see that $C_i,S_i$ are normal.
By adjunction, 
we can write $(K_{X_i}+B_i)|_{S_i}=K_{S_i}+B_{S_i}$ where 
the coefficient of $C_i$ in $B_{S_i}$ is one. Applying adjunction once more, 
we can write the pullback of $K_{S_i}+B_{S_i}$ to $C_i$ as $K_{C_i}+B_{C_i}$
for some boundary $B_{C_i}$. Since $C_i\simeq {C_{i+1}}$, we will use the notation 
$(C,B_{i,C})$ instead of $({C_i},B_{C_i})$. 
Since each step of the LMMP makes the divisor $K_X+B$ "smaller", 
$$
K_{C}+B_{i,C}\ge K_{C}+B_{i+1,C}
$$  
hence $B_{i,C}\ge B_{i+1,C}$ for every $i$. By Propositions 
\ref{p-adjunction-DCC}, the 
coefficients of $B_{S_i}$ and $B_{i,C}$ belong to some fixed DCC set. 
Therefore $B_{i,C}=B_{i+1,C}$ for every $i\gg 0$ which implies that 
after finitely many steps, each remaining step of the LMMP 
is an isomorphism near $C_i$.   

From now on we may assume that all the steps of the LMMP are flips.
Let $S$ be any lc centre of dimension $2$, i.e. a component of $B$ with coefficient one.  
If $S_i$ intersects the exceptional locus $E_i$ of $X_i\to Z_i$, then no other component of $\rddown{B_i}$ 
can intersect the exceptional locus: assume that another component $T_i$ intersects the exceptional 
locus; if either $S_i$ or $T_i$ contains  $E_i$, then $S_i\cap T_i$ intersects $E_i$; 
but $S_i\cap T_i$ is a union of lc centres of dimension one and this contradicts the 
last paragraph; so none of $S_i,T_i$ contains $E_i$. But then both contain the exceptional 
locus of $X_{i+1}\to Z_i$ and similar arguments give a contradiction. 

Assume $D_i\subset S_i$ is a component of the exceptional locus of  
$X_i\to Z_{i-1}$ where $i>1$. Then the log discrepancy of $D_i$ with 
respect to $(S_1,B_{S_1})$ is less than one. Moreover, we can assume that the generic point of the centre of $D_i$ 
on $S_1$ is inside the klt locus of $(S_1,B_{S_1})$ by the last paragraph. But there can be at most finitely 
many such $D_i$ (as prime divisors on birational models of $S_1$). 
Since the coefficients of $D_i$ in $B_{S_i}$ belongs to a DCC set, 
the coefficient of $D_i$ stabilizes. Therefore after finitely many steps, $S_i$ 
cannot contain any component of the exceptional locus of $X_i\to Z_{i-1}$. 
So we get a sequence $S_i\bir S_{i+1}$ of birational morphisms which are isomorphisms 
 if $i\gg 0$. In particular, $S_i$ is disjoint from $E_i$ for $i\gg 0$.\\
\end{proof}

\section{Existence of log flips}\label{s-flips}

In this section, we first prove that generalized flips exist (\ref{t-flip-2}). 
Next we prove  Theorem \ref{t-flip-1} 
in the projective case, that is, when $X$ is projective. The general case of 
 Theorem \ref{t-flip-1} is proved in Section \ref{s-mmodels} where $X$ is quasi-projective.

\subsection{Divisorial and flipping extremal rays}
Let $(X,B)$ be a projective $\Q$-factorial pair of dimension $3$ over $k$ of char $p>0$, and 
let $R$ be a $K_X+B$-negative extremal ray. Assume that there is a nef and big $\Q$-divisor 
$L$ such that $R=L^\perp$.  We say $R$ is a 
\emph{divisorial extremal ray} if  $\dim \mathbb{E}(L)=2$. But we say  
$R$ is a \emph{flipping extremal ray} if $\dim \mathbb{E}(L)=1$. By \ref{ss-ext-rays-II}, 
such rays can be contracted to algebraic spaces. By \ref{ss-ext-rays}, when $K_X+B$ is pseudo-effective, 
 each $K_X+B$-negative extremal ray is either a divisorial extremal ray or a 
flipping extremal ray.
 We will show below (\ref{t-contraction}) 
that any divisorial or flipping extremal ray can actually be contracted by a projective 
morphism if $(X,B)$ is dlt and $p>5$. However, we still need contractions to algebraic spaces as an auxiliary tool.

\subsection{Existence of generalized flips}

We recall the definition of \emph{generalized flips} which was introduced in [\ref{HX}]. 
Let $(X,B)$ be a projective $\Q$-factorial pair of dimension $3$ over $k$ char $p>0$, and 
let $R$ be a $K_X+B$-negative flipping extremal ray. We say that the {generalized 
flip} of $R$ exists (see [\ref{HX}, after Theorem 5.6]) 
if there is a birational map $X\bir X^+/V$ which is an isomorphism in 
codimension one, $X^+$ is $\Q$-factorial projective, and $K_{X^+}+B^+$ is numerically 
positive on any curve contracted by $X^+\to V$.

\begin{thm}\label{t-flip-2}
Let $(X,B)$ be a projective $\Q$-factorial dlt pair of dimension $3$ over $k$ of char $p>5$.  
Let $R$ be a $K_X+B$-negative flipping extremal ray. Then 
the generalized flip of $R$ exists. 
\end{thm}

The theorem  was proved in [\ref{HX}, Theorem 5.6] when $B$ has standard coefficients and 
$K_X+B$ is pseudo-effective. 

\begin{proof} 
This proof (as well as the proof of [\ref{HX}, Theorem 5.6]) is modeled on the proof of Shokurov's 
reduction theorem [\ref{Shokurov-pl}, Theorem 1.2].
Since $R$ is a flipping extremal ray, by definition, there is a nef and big $\Q$-divisor 
$L$ such that $R=L^\perp$. Moreover, $L$ is endowed with a map $X\to V$ to an algebraic space 
which contracts the curves generating $R$. 
Note that if $B'$ is another boundary such that $(K_X+B')\cdot R<0$, then the generalized flip 
exists for $(X,B)$ if and only if it exists for $(X,B')$. This follows from the fact that 
$K_X+B\equiv t(K_X+B')/V$ for some number $t>0$ where the numerical equivalence means that 
$K_X+B-t(K_X+B')$ is numerically trivial on any curve contracted by $X\to V$.

Let $\mathfrak{S}$ be the set of standard coefficients as defined in the introduction.
Define 
$$
\zeta(X,B)=\#\{S \mid \mbox{$S$ is a component of $B$ and its coefficient is not in~} \mathfrak{S}\}
$$ 
Assume that the generalized flip of $R$ does not exist. We will derive a contradiction.
We can assume that $\zeta(X,B)$ 
is minimal, that is, we may assume that generalized flips always exist for 
pairs with smaller $\zeta$. We can decrease the coefficients of $\rddown{B}$ slightly so that 
$(X,B)$ becomes klt and $\zeta(X,B)$ is unchanged. In addition, each 
component $S$ of $B$ whose coefficient is not in $\mathfrak{S}$ satisfies $S\cdot R<0$ 
otherwise we can discard $S$ and decrease $\zeta(X,B)$ which is not possible by the 
minimality assumption.

First assume that $\zeta(X,B)>0$.
 Choose a component $S$ of $B$ whose coefficient $b$ is not 
in $\mathfrak{S}$. There is a positive number $a$ such that $K_X+B\equiv aS/V$.
Let $g\colon W\to X$ be a log resolution, and let $B_W=B^\sim+E$ 
and $\Delta_W=B_W+(1-b)S^\sim$
where $E$ is the reduced exceptional divisor of $g$ and $B^\sim,S^\sim$ are 
birational transforms. Note that $\rddown{B_W}=E$ and $\rddown{\Delta_W}=S^\sim+E$.
Since $(X,B)$ is klt, 
$$
K_W+\Delta_W=K_W+B_W+(1-b)S^\sim=g^*(K_X+B)+G+(1-b)S^\sim
$$
where $G$ is effective and its support is equal to the support of $E$. Thus
$$
K_W+\Delta_W\equiv g^*(aS)+G+(1-b)S^\sim=(a+1-b)S^\sim+F/V
$$
where $F$ is effective and $\Supp F=\Supp E$. By construction, we have 
$$
\mbox{$\Supp (S^\sim+F)=\rddown{\Delta_W}~~$ and $~~\zeta(W,\Delta_W)<\zeta(X,B)$}
$$

Run an LMMP$/V$ on $K_W+\Delta_W$ with scaling of some ample divisor, as in \ref{ss-g-LMMP-scaling}. 
Recall that this is an LMMP$/\mathcal{C}$ 
on $K_W+\Delta_W$ where $\mathcal{C}=N^\perp$ and $N$ is the pullback of the nef and big $\Q$-divisor $L$.
In each step some component of $\rddown{\Delta_W}$ 
is negative on the corresponding extremal ray. So such extremal rays are pl-extremal rays, 
 they can be contracted by projective 
morphisms, and the $\Q$-factorial property is preserved  (see \ref{ss-pl-ext-rays}). 
Moreover, if we encounter a flipping contraction, then 
its generalized flip exists because 
$\zeta(W,\Delta_W)<\zeta(X,B)$ and because we chose $\zeta(X,B)$ to be minimal; 
the flip is a usual one since its extremal ray is contracted projectively. 
By special termination (\ref{p-st}), 
the LMMP terminates on some model $Y/V$.

Now run an LMMP$/V$ on $K_Y+B_Y$ with scaling of $(1-b)S_Y$  
where $B_Y$ is the pushdown of 
$B_W$ and $S_Y$ is the pushdown of $S^\sim$. Since we have the numerical equivalence  
$K_Y+B_Y\equiv a S_Y+F_Y/V$ and $\Supp F_Y=\rddown{B_Y}$,
in each step of the LMMP the corresponding extremal ray intersects 
some component of $\rddown{B_Y}$ negatively hence they are pl-extremal rays and 
they can be contracted by projective morphisms (\ref{ss-pl-ext-rays}). Moreover, 
if one of these rays gives a  
flipping contraction, then its generalized flip exists because $K_Y+B_Y-bS_Y$ 
is negative on that ray and $\zeta(Y,B_Y-bS_Y)<\zeta(X,B)$. 
Note that again such flips are usual flips.
The LMMP terminates on a model $X^+$ by special termination. 

Let $h\colon W'\to X$ and $e\colon W'\to X^+$ be a common resolution.
Now the negativity lemma (\ref{ss-negativity}) 
applied to the divisor $h^*(K_X+B)-{e^*}(K_{X^+}+B^+)$ over $X$ implies  
that 
$$
h^*(K_X+B)-{e^*}(K_{X^+}+B^+)\ge 0
$$ 
Thus  
every component $D$ of $E$ is contracted over $X^+$ because  
$$
0<a(D,X,B)\le a(D,X^+,B^+)
$$
Therefore $X\bir X^+$ is an isomorphism in codimension one. It is enough to show that  
$K_{X^+}+B^+$ is numerically positive$/V$. Let $H^+$ be an ample 
divisor on $X^+$ and $H$ its birational transform on $X$. There is a positive number 
$c$ such that $K_X+B\equiv cH/V$ hence $K_{X^+}+B^+\equiv cH^+/V$ which implies that 
$K_{X^+}+B^+$ is numerically positive$/V$. 
So we have constructed the generalized flip and this contradicts our assumptions above.

Now assume that $\zeta(X,B)=0$. If $K_X+B$ is pseudo-effective, then we can simply 
apply [\ref{HX}, Theorem 5.6] to get a contradiction. Unfortunately, $K_X+B$ may not be pseudo-effective (note that 
even if we originally start with a pseudo-effective log divisor we may end up with a 
non-pseudo-effective $K_X+B$ since we decreased some coefficients). However, this is not 
a problem because the proof of [\ref{HX}, Theorem 5.6] still works. Since there is a 
nef and big $\Q$-divisor $L$ with $L\cdot R=0$, there is a prime divisor $S$ with $S\cdot R<0$. 
There is a number $a>0$ such that $K_X+B\equiv aS/V$. Now take a 
log resolution $g\colon W\to X$ and define $B_W$ and $\Delta_W$ as above (if $S$ is not 
a component of $B$ simply let $b=0$). Run an LMMP$/V$ 
on $K_W+\Delta_W$. The extremal rays in the process are all pl-extremal rays hence 
they can be contracted by projective morphisms. Moreover, if we encounter a 
flipping contraction, then its flip exists by [\ref{HX}, Theorem 4.12] because  
all the coefficients of $\Delta_W$ are standard. The LMMP terminates on some model $Y$ by the 
special termination. Next, run the LMMP$/V$ on $K_Y+B_Y$ with scaling of $(1-b)S_Y$. 
Again, the extremal rays in the process are all pl-extremal rays hence 
they can be contracted by projective morphisms. Moreover, if we encounter a 
flipping contraction, then its flip exists by [\ref{HX}, Theorem 4.12] because 
all the coefficients of $B_Y$ are standard.
The LMMP terminates on some model $X^+$ by the 
special termination. The rest of the argument goes as before.\\
\end{proof}

\subsection{Proof of \ref{t-flip-1} in the projective case}
\begin{proof}(of Theorem \ref{t-flip-1} in the projective case)
Assume that $X$ is projective. Then by Theorem \ref{t-flip-2}, the generalized 
flip of the extremal ray of $X\to Z$ exists. But since $X\to Z$ is a projective contraction, 
the generalized flip is a usual flip. 

If $X$ is only quasi-projective, we postpone the proof to Section \ref{s-mmodels}. Until then we need 
flips only in the projective case.\\
\end{proof}

\section{Crepant models}\label{s-crepant-models}

\subsection{Divisorial extremal rays}
The next lemma is essentially {[\ref{HX}, Theorem 5.6(2)]}. 

\begin{lem}\label{l-div-ray}
Let $(X,B)$ be a projective $\Q$-factorial dlt pair of dimension $3$ over $k$ of char $p>5$. 
Let $R$ be a $K_X+B$-negative divisorial extremal ray. Then $R$ can be 
contracted by a projective morphism $X\to Z$ where $Z$ is $\Q$-factorial.
\end{lem}
\begin{proof}
We may assume that $(X,B)$ is klt. 
Since $R$ is a divisorial extremal ray, by definition, there is a nef and big $\Q$-divisor 
$L$ such that $R=L^\perp$ and $\dim \mathbb{E}(L)=2$. 
Moreover, $R$ can be contracted by a map $X\to V$ to an algebraic space. 
There is a prime divisor $S$ with $S\cdot R<0$. In particular, $\mathbb{E}(L)\subseteq S$ and 
$S$ is the only prime divisor contracted by $X\to V$. 
There is a number $a>0$ such that $K_X+B\equiv aS/V$. 
Let $g\colon W\to X$ be a log resolution and define  $\Delta_W$ as in the 
proof of Theorem \ref{t-flip-2}. Run an LMMP$/V$ on $K_W+\Delta_W$. As in \ref{t-flip-2}, 
the extremal rays in the process are pl-extremal rays hence they are contracted projectively 
and the LMMP terminates with a model $Z$.
We are done if we show that $Z\to V$ is an isomorphism (the $\Q$-factoriality claim follows from 
\ref{ss-pl-ext-rays}). Assume this is not the case.

Recall that 
$$
K_W+\Delta_W\equiv (a+1-b)S^\sim+F/V
$$ 
and now $(a+1-b)S^\sim+F$ is exceptional$/V$. In particular, $(a+1-b)S_Z+F_Z$ is 
effective, exceptional and nef$/V$. 

Let $H_Z$ be a general ample divisor on $Z$ 
and $H$ its birational transform on $X$. There is a number $t\ge 0$ such that 
$H+tS\equiv 0/V$. Therefore there is an effective and exceptional$/V$ divisor 
$P_Z$ such that $H_Z+P_Z\equiv 0/V$. Note that $\Supp P_Z$ contains all the exceptional 
divisors of $Z\to V$ hence $\Supp P_Z=\Supp F_Z$. Moreover, $P_Z\neq 0$ otherwise $H_Z\equiv 0/V$ 
hence $Z\to V$ is an isomorphism which is not the case by assumption. This also shows that 
$F_Z\neq 0$.

Let $s$ be the smallest number such that 
$$
Q_Z:=(a+1-b)S_Z+F_Z-sP_Z\le 0
$$
Then $Q_Z$ is numerically positive over $V$ and there is some prime exceptional$/V$ 
divisor $D$ which is not a component of $Q_Z$. This is not possible since 
$Q_Z$ cannot be numerically positive on the general curves of $D$ contracted$/V$.\\ 
\end{proof}

\subsection{Projectivization and dlt models}

\begin{lem}\label{l-reduced-Cartier}
Let $X$ be a normal projective variety over $k$ and $D\neq X$ a closed subset. Then 
there is a reduced effective Cartier divisor $H$ whose support contains $D$.  
\end{lem}
\begin{proof}
We may assume that each irreducible component of $D$ is a prime divisor hence we 
can think of $D$ as a reduced Weil divisor. Let $A$ be a sufficiently ample divisor. 
Let $U$ be the smooth locus of $X$. Since $(A-D)|_U$ is sufficiently ample, 
we can choose a reduced effective divisor $H'$ with no common components with $D$ 
such that $H'|_U\sim (A-D)|_U$. This 
extends to $X$ and gives $H'\sim A-D$. Now $H:=H'+D\sim A$ is Cartier and satisfies the 
requirements.\\ 
\end{proof}

The next few results are standard consequences of special termination (cf. 
[\ref{B-mmodel}, Lemma 3.3][\ref{HX}, Theorem 6.1]).

\begin{lem}\label{l-dlt-model-proj}
Let $(X,B)$ be an lc pair of dimension $3$ over $k$ of char $p>5$, and let $\overline{X}$ be a projectivization of  
$X$. Then there is a projective $\Q$-factorial dlt pair $(\overline{Y},B_{\overline{Y}})$ 
with a birational morphism $\overline{Y}\to \overline{X}$
satisfying the following: 

$\bullet$ $K_{\overline{Y}}+B_{\overline{Y}}$ is nef$/\overline{X}$,

$\bullet$ let $Y$ be the inverse image of $X$ and $B_Y=B_{\overline{Y}}|_Y$; 
 then $(Y,B_Y)$ is a $\Q$-factorial dlt model of $(X,B)$.
\end{lem}
\begin{proof}
We may assume that $\overline{X}$ is normal.
By Lemma \ref{l-reduced-Cartier}, there is a reduced effective Cartier divisor $H$ containing the complement of $X$ in 
$\overline{X}$. We may assume that $H$ has no common components with $B$. 
Let $f\colon \overline{W}\to \overline{X}$ be a log resolution. 
Now let $B_{\overline{W}}$ be the sum of the reduced exceptional divisor of $f$ and the birational 
transform of $B$, and let $\Delta_{\overline{W}}$ be the sum of $B_{\overline{W}}$ and the birational 
transform of $H$.

Run the LMMP$/\overline{X}$ on $K_{\overline{W}}+\Delta_{\overline{W}}$ inductively as follows. 
Assume that we have arrived at a model $\overline{Y}$. Let $R$ be a 
$K_{\overline{Y}}+\Delta_{\overline{Y}}$-negative extremal ray$/\overline{X}$. 
Let $\overline{Y}\to \overline{Z}$ be the contraction of $R$ 
to an algebraic space, and let $L$ be a nef and big $\Q$-divisor with $L^\perp=R$. Any curve contracted 
by $\overline{Y}\to \overline{Z}$  is also contracted over $\overline{X}$.  
If $\dim \mathbb{E}(L)=2$, then $R$ is a 
divisorial extremal ray hence  $\overline{Y}\to \overline{Z}$ is a projective contraction 
by Lemma \ref{l-div-ray}. In this case, we continue the program with $\overline{Z}$. 
Now assume that $\dim \mathbb{E}(L)=1$.
 Let $C$ be a connected component of 
$\mathbb{E}(L)$ and $P$ its image in $\overline{X}$ which is just a point. 
If $P\in \Supp H$, then $C$ is contained in some  
component of the pullback of $H$ hence it is contained in some component of 
$\rddown{\Delta_{\overline{Y}}}$. In this case, $\overline{Y}\to \overline{Z}$ 
is again a projective contraction by \ref{ss-good-exc-locus}. Now assume that $P$ does not belong to the 
support of $H$. Since $(X,B)$ is lc, over $X\setminus H$ the divisor
$$
K_{\overline{W}}+\Delta_{\overline{W}}-f^*(K_X+B)
$$
is effective and exceptional$/\overline{X}$ hence some component of $\Delta_{\overline{Y}}$ intersects 
$R$ negatively which implies again that the contraction $\overline{Y}\to \overline{Z}$ is projective. 
Therefore in any case $R$ can be contracted by a projective morphism and we can continue the LMMP as usual. 
The required flips exist by the results of Section \ref{s-flips}.
By special termination (\ref{p-st}), the LMMP terminates say on ${{\overline{Y}}}$. 

Next, we run the LMMP$/\overline{X}$ on $K_{\overline{Y}}+B_{\overline{Y}}$ with scaling of 
$\Delta_{\overline{Y}}-B_{\overline{Y}}$ as in \ref{ss-g-LMMP-scaling}. Note that $\Delta_{\overline{Y}}-B_{\overline{Y}}$ is 
nothing but the 
birational transform of $H$. Since the pullback of $H$ is numerically trivial over $\overline{X}$, 
each extremal ray in the process intersects some exceptional divisor negatively hence 
such extremal rays can be contracted by projective morphisms. Moreover, the required 
flips exist and by special termination the LMMP 
terminates on a model which we may again denote by ${\overline{Y}}$. Now let $Y$ 
be the inverse image of $X$ under $g\colon \overline{Y}\to \overline{X}$ and let $B_Y$ 
be the restriction of $B_{\overline{Y}}$ to $Y$. Then $(Y,B_Y)$ is a $\Q$-factorial 
dlt model of $(X,B)$ because $K_Y+B_Y-g^*(K_X+B)$ is effective and exceptional hence zero 
as it is nef$/X$.\\
\end{proof}

\begin{proof}(of Theorem \ref{cor-dlt-model})
This is already proved in Lemma \ref{l-dlt-model-proj}.\\
\end{proof}

\subsection{Extraction of divisors and terminal models}

\begin{lem}\label{l-extraction}
Let $(X,B)$ be an lc pair of dimension $3$ over $k$ of char $p>5$ and let $\{D_i\}_{i\in I}$ be a finite set 
of exceptional$/X$ prime divisors (on birational
 models of $X$) such that $a(D_i,X,B)\le 1$. Then there is a $\Q$-factorial dlt pair $(Y,B_Y)$ 
 with a projective birational morphism $Y\to X$ such that\\
 $\bullet$ $K_Y+B_Y$ is the crepant pullback of $K_X+B$,\\
  $\bullet$  every exceptional/$X$ prime divisor $E$ of $Y$ is one of the $D_i$ or $a(E,X,B)=0$,\\
  $\bullet$  the set of exceptional/$X$ prime divisors of $Y$ includes $\{D_i\}_{i\in I}$.
\end{lem}
\begin{proof}
By Lemma \ref{l-dlt-model-proj}, we can assume that $(X,B)$ is projective 
$\Q$-factorial dlt. Let $f\colon W\to X$ be a log resolution and let $\{E_j\}_{j\in J}$
be the set of prime exceptional divisors of $f$. We can assume that
for some $J'\subseteq J$, $\{E_j\}_{j\in J'}=\{D_i\}_{i\in I}$.
Now define
$$
K_W+B_W:=f^*(K_{X}+B)+\sum_{j\notin J'} a(E_j,X,B)E_j
$$
which ensures that if $j\notin J'$, then $E_j$ is a component of $\rddown{B_W}$.
Run an LMMP$/X$ on $K_W+B_W$ which would be an LMMP on $\sum_{j\notin J'} a(E_j,X,B)E_j$. 
So each extremal ray in the process intersects some component of $\rddown{B_W}$ negatively 
hence such rays can be contracted by projective morphisms (\ref{ss-pl-ext-rays}), 
the required flips exists (Section \ref{s-flips}), 
and the LMMP terminates by special termination (\ref{p-st}), say on a model $Y$. 
Now $(Y,B_Y)$ satisfies all the requirements.\\  
\end{proof}

\begin{proof}(of Corollary \ref{cor-terminal-model})
Apply Lemma \ref{l-extraction} by taking $\{D_i\}_{i\in I}$ to be the set of all 
prime divisors with log discrepancy $a(D_i,X,B)\le 1$.\\
\end{proof}

\section{Existence of log minimal models}\label{s-mmodels}

\subsection{Weak Zariski decompositions}\label{ss-WZD}
Let $D$ be an $\R$-Cartier divisor on a normal variety $X$ and $X\to Z$ a projective contraction 
over $k$. 
A \emph{weak Zariski decomposition$/Z$} for $D$
consists of a projective birational morphism $f\colon W\to X$ from a normal variety, and a numerical equivalence 
$f^*D\equiv P+M/Z$ such that 
\begin{enumerate}
\item $P$ and $M$ are $\R$-Cartier divisors, 
\item $P$ is nef$/Z$, and $M\ge 0$.
\end{enumerate}
We then define ${\theta}(X,B,M)$ to be the number of those components of $f_*M$ which are not components of 
$\rddown{B}$.

\subsection{From weak Zariski decompositions to minimal models}
We use the methods of [\ref{B-WZD}], which is somewhat similar to [\ref{BCHM}, \S 5], 
to prove the following result.

\begin{prop}\label{p-WZD}
Let $(X,B)$ be a projective lc pair of dimension $3$ over $k$ of char $p>5$, 
and $X\to Z$ a projective contraction. 
Assume that $K_X+B$ has a weak Zariski decomposition$/Z$. Then $(X,B)$ has a 
log minimal model over $Z$.
\end{prop}
\begin{proof}
Assume that $\mathfrak{W}$ is the set of pairs $(X,B)$ and projective contractions $X\to Z$ 
such that\\ 
\begin{description}
\item[L] $(X,B)$ is projective, lc of dimension $3$ over $k$,  
\item[Z] $K_X+B$ has a weak Zariski decomposition$/Z$, and 
\item[N] $(X,B)$ has no log minimal model over $Z$.\\
\end{description}

Clearly, it is enough to show that $\mathfrak{W}$ is empty. Assume otherwise and let $(X,B)$ and 
$X\to Z$ be in $\mathfrak{W}$.  
Let $f\colon W\to X$, $P$ and $M$ be the data given by 
a weak Zariski decomposition$/Z$ for $K_X+B$ as in \ref{ss-WZD}. Assume in addition that 
${\theta}(X,B,M)$ is minimal. Perhaps after replacing $f$ we can assume that $f$ gives a 
log resolution of $(X, \Supp (B+f_*M))$. Let $B_W=B^\sim+E$ where $B^\sim$ is the birational 
transform of $B$ and $E$ is the reduced exceptional divisor of $f$. Then 
$$
K_W+B_W=f^*(K_X+B)+F\equiv P+M+F/Z
$$
is a weak Zariski decomposition where $F\ge 0$ is exceptional$/X$. Moreover, 
$$
{\theta}(W,B_W,M+F)={\theta}(X,B,M)
$$ 
and any log minimal model of 
$(W,B_W)$ is also a log minimal model of $(X,B)$ [\ref{B-WZD}, Remark 2.4].
So by replacing $(X,B)$ with $(W,B_W)$ and $M$ with $M+F$ we may assume 
that $W=X$, $(X,\Supp (B+M))$ is log smooth, and that $K_X+B\equiv P+M/Z$.

First assume that ${\theta}(X,B,M)=0$, that is, $\Supp M\subseteq \rddown{B}$. 
Run the LMMP$/Z$ on $K_X+B$ using $P+M$ 
as in \ref{ss-LMMP-WZD}. Obviously, $M$ negatively intersects each 
extremal ray in the process, and since $\Supp M\subseteq \rddown{B}$, 
the rays are pl-extremal rays. Therefore those rays can be contracted by projective morphisms 
(\ref{ss-pl-ext-rays}), the required flips exist (Section \ref{s-flips}), and the LMMP terminates by special 
termination (\ref{ss-special-termination}). Thus we get a log minimal model of $(X,B)$ over $Z$ which 
contradicts the assumption that $(X,B)$ and $X\to Z$ belong to $\mathfrak{W}$. 
For the rest of the proof we do not use LMMP.

From now on we assume that ${\theta}(X,B,M)>0$. Define
$$
\alpha:=\min\{t>0~|~~\rddown{(B+tM)^{\le 1}}\neq \rddown{B}~\}
$$
where for a divisor $D=\sum d_iD_i$ we define $D^{\le 1}=\sum d_i'D_i$ with $d_i'=\min\{d_i,1\}$.
 In particular, $(B+\alpha M)^{\le 1}=B+C$ for some $C\ge 0$ 
supported in $\Supp M$, and $\alpha M=C+A$ 
where $A\ge 0$ is supported in $\rddown{B}$ and $C$ has no common components with $\rddown{B}$. Note that  
${\theta}(X,B,M)$ is equal to the number of components of $C$.
The pair $(X,B+C)$ is  lc and the expression 
$$
K_X+B+C\equiv P+M+C/Z
$$ 
is a weak Zariski decomposition$/Z$. By construction
$$
{\theta}(X,B+C,M+C)<{\theta}(X,B,M)
$$
so $(X,B+C)$ has a log minimal model over $Z$ by minimality of ${\theta}(X,B,M)$ 
and the definition of $\mathfrak W$. Let $(Y,(B+C)_Y)$ be the minimal model. 

Let $g\colon V\to X$ and $h\colon V\to Y$ be a common resolution. 
By definition, $K_Y+(B+C)_Y$ is nef/$Z$. 
In particular, the expression 
$$
g^*(K_X+B+C)= P'+M'
$$ 
is a weak Zariski decomposition$/Z$ of $K_X+B+C$ where $P'=h^*(K_Y+(B+C)_Y)$ 
and $M'\ge 0$ is exceptional$/Y$ (cf. [\ref{B-WZD}, Remark 2.4 (2)]). 
Moreover, 
$$
g^*(K_X+B+C)=P'+M'\equiv g^*P+g^*(M+C)/Z
$$
Since $M'$ is exceptional$/Y$, 
$$
h_*(g^*(M+C)-M')\ge 0
$$ 
On the other hand, 
$$
g^*(M+C)-M'\equiv P'-g^*P/Z
$$ 
is anti-nef$/Y$ hence by the negativity lemma, 
$g^*(M+C)-M'\ge 0$.
Therefore $\Supp M'\subseteq \Supp g^*(M+C)=\Supp g^*M$. 

Now, 

\begin{equation*}
\begin{split}
(1+\alpha)g^*(K_X+B) & \equiv g^*(K_X+B)+\alpha g^*P+ \alpha g^*M\\
& \equiv g^*(K_X+B)+\alpha g^*P+g^* C+g^*A\\
& \equiv P'+\alpha g^*P+M'+g^*A/Z
\end{split}
\end{equation*}
hence we get a weak Zariski decomposition$/Z$ as 
$$
g^*(K_X+B)\equiv P''+M''/Z
$$ 
where 
$$
P''=\frac{1}{1+\alpha}(P'+ \alpha g^*P)  \mbox{\hspace{0.5cm}  and \hspace{0.5cm}}  M''=\frac{1}{1+\alpha}(M'+g^*A)
$$
and $\Supp M''\subseteq \Supp g^* M$ hence $\Supp g_*M''\subseteq \Supp M$. Since ${\theta}(X,B,M)$ is minimal,  
$$
{\theta}(X,B,M)={\theta}(X,B,M'')
$$ 
 So every component of $C$ is also a component 
of $g_*M''$ which in turn implies that every component of $C$ is also a 
component of $g_*M'$. But $M'$ is exceptional$/Y$ 
hence so is $C$ which means that $(B+C)_Y=B^\sim+C^\sim+E=B^\sim+E=B_Y$ where $\sim$ stands for birational transform 
and $E$ is the reduced exceptional divisor of $Y\bir X$. Thus we have $P'=h^*(K_Y+B_Y)$. Although $K_Y+B_Y$ 
is nef$/Z$, $(Y,B_Y)$ is not necessarily a 
log minimal model of $(X,B)$ over $Z$ because condition (4) of definition 
of log minimal models may not be satisfied (see \ref{ss-mmodels}). 

Let $G$ be the largest $\R$-divisor such that $G\le g^*C$ and $G\le {M}'$. By letting $\tilde{C}=g^*C-G$ and 
$\tilde{M}'= M'-G$ we get the expression 
$$
g^*(K_X+B)+\tilde{C}= P'+\tilde{M}'
$$ 
where $\tilde{C}$ and $\tilde{M}'$ are effective with no common components. 

 Assume that $\tilde{C}$ is exceptional$/X$.  Then $g^*(K_X+B)-P'=\tilde{M}'-\tilde{C}$ is antinef$/X$ 
so by the negativity lemma $\tilde{M}'-\tilde{C}\ge 0$ which implies that $\tilde{C}=0$ since 
$\tilde{C}$ and $\tilde{M}'$ 
have no common components. Thus
$$
g^*(K_X+B)-h^*(K_Y+B_Y)=\sum_D a(D,Y,B_Y)D-a(D,X,B)D=\tilde{M}'
$$
where $D$ runs over the prime divisors on $V$.
If $\Supp g_*\tilde{M}'=\Supp g_*M'$, then $\Supp \tilde{M}'$ contains the birational transform of 
all the prime exceptional$/Y$ divisors on $X$ 
hence $(Y,B_Y)$ is a log minimal model of $(X,B)$ over $Z$, a contradiction. Thus 
$$
\Supp (g_*M'-g_*G)=\Supp g_*\tilde{M}'\subsetneq \Supp g_* M'\subseteq \Supp M
$$ 
so some component of $C$ is not a component of $g_*\tilde{M}'$ because $\Supp g_*G\subseteq \Supp C$. Therefore 
$$
{\theta}(X,B,M)>{\theta}(X,B,\tilde{M}')
$$ 
which gives a contradiction again by minimality of ${\theta}(X,B,M)$ and the assumption
 that $(X,B)$ has no log minimal model over $Z$.

 So we may assume that $\tilde{C}$ is not exceptional$/X$. Let 
$\beta>0$ be the smallest number such that $\tilde{A}:=\beta g^* M-\tilde{C}$ satisfies $g_*\tilde{A}\ge 0$. Then 
there is a component of $g_*\tilde{C}$ which is not a component of $g_*\tilde{A}$. 
Now 

\begin{equation*}
\begin{split}
(1+\beta)g^*(K_X+B) & \equiv g^*(K_X+B)+\beta g^*M+\beta g^*P\\
&  \equiv g^*(K_X+B)+\tilde{C}+\tilde{A}+\beta g^*P\\
& \equiv P'+\beta g^*P+ \tilde{M}'+\tilde{A}/Z
\end{split}
\end{equation*}
where $\tilde{M}'+\tilde{A}\ge 0$ by the negativity lemma. 
Thus we get a weak Zariski decomposition$/Z$ as $g^*(K_X+B)\equiv P'''+M'''/Z$
 where 
$$
P'''=\frac{1}{1+\beta}(P'+\beta g^*P)  \mbox{\hspace{0.5cm}  and \hspace{0.5cm}}  M'''=\frac{1}{1+\beta}(\tilde{M}'+\tilde{A})
$$
and $\Supp g_*M'''\subseteq \Supp M$.
Moreover, by construction, there is a component $D$ of $g_*\tilde{C}$ which is not a component of $g_*\tilde{A}$. 
Since $g_*\tilde{C}\le C$, $D$ is a component of $C$ hence of $M$, and since $\tilde{C}$ and $\tilde{M}'$ 
have no common components, $D$ is not a component of $g_*\tilde{M}'$. Therefore $D$ is not a component of 
$g_*M'''=\frac{1}{1+\beta}(g_*\tilde{M}'+g_*\tilde{A})$ which implies that   
$$
{\theta}(X,B,M)>{\theta}(X,B,M''')
$$ 
giving a contradiction again.\\ 
\end{proof}

\subsection{Proofs of \ref{t-mmodel} and \ref{t-flip-1}}

\begin{proof}(of Theorem \ref{t-mmodel})
By applying Lemma \ref{l-dlt-model-proj}, we can reduce the problem to the 
case when $X,Z$ are projective. We can find a log resolution $f\colon W\to X$ 
and a $\Q$-boundary $B_W$ such that  
$$
K_W+B_W=f^*(K_X+B)+E
$$ 
where $E\ge 0$ and its support is equal to the union of the 
exceptional divisors of $f$, and $(W,B_W)$ has terminal singularities. It is enough to 
construct a log minimal model for $(W,B_W)$ over $Z$. So by replacing $(X,B)$ 
with $(W,B_W)$ we can assume $(X,B)$ has terminal singularities and that $X$ is $\Q$-factorial.
 
 Let 
$$
\mathcal{E}=\{B' \mid K_X+B' ~~\mbox{is pseudo-effective$/Z$ and $0\le B'\le B$}\}
$$
which is a compact subset of the $\R$-vector space $V$ generated by the components of $B$. 
Let $B'$ be an element in $\mathcal{E}$ which has minimal 
distance from $0$ with respect to the standard metric on $V$. So either $B'=0$, or 
$K_X+B''$ is not pseudo-effective$/Z$ for any $0\le B''\lneq B'$.
    
Run the generalized LMMP$/Z$ on $K_X+B'$ as follows [\ref{HX}, proof of Theorem 5.6]: 
let $R$ be a $K_{X}+B'$-negative extremal ray$/Z$. By \ref{ss-ext-rays-II}, $R$ is either 
a divisorial extremal ray or a flipping extremal ray (see the beginning of Section \ref{s-flips} 
for definitions), and $R$ can be contracted to an algebraic space. 
If $R$ is a divisorial extremal ray, then it can actually be contracted by a projective 
morphism, by Lemma \ref{l-div-ray}, and we continue the process. But if $R$ is a 
flipping extremal ray, then we use the generalized flip, which exists by Theorem \ref{t-flip-2}, 
and then continue the process.
 
No component of $B'$ is contracted by the LMMP: otherwise let $X_i\bir X_{i+1}$ 
be the sequence of log flips and divisorial contractions of this LMMP where $X=X_1$. 
Pick $j$ so that $\phi_j\colon X_j\bir X_{j+1}$ is a divisorial contraction 
which contracts a component $D_j$ of $B_j'$, the birational transform of $B'$. Now there is $a>0$ 
such that 
$$
K_{X_j}+B_j'=\phi_j^*(K_{X_{j+1}}+B_{j+1}')+aD_j
$$ 
Since $K_{X_{j+1}}+B_{j+1}'$ is pseudo-effective$/Z$, 
$K_{X_{j}}+B_{j}'-aD_j$ is pseudo-effective$/Z$ which implies that $K_X+B'-bD$ is pseudo-effective$/Z$ 
for some $b>0$ where 
$D$ is the birational transform of $D_j$, a contradiction.
Therefore every $(X_j,B'_j)$ has terminal singularities. 
The LMMP terminates for reasons similar to the characteristic $0$ case 
[\ref{Shokurov-nv}, Corollary 2.17][\ref{Kollar-Mori}, Theorem 6.17] (see also [\ref{HX}, proof of Theorem 1.2]).  
So we get a log minimal model of $(X,B')$ over $Z$, say $(Y,B'_{Y})$. 

Let $g\colon V\to X$ and $h\colon V\to Y$ be a common resolution. By letting 
$P=h^*(K_{Y}+B'_{Y})$ and 
$$
M=g^*(K_X+B)-h^*(K_{Y}+B'_{Y})
$$ 
we get a weak Zariski decomposition$/Z$ as 
$
g^*(K_X+B)=P+M/Z.
$
Note that $M\ge 0$ because $g^*(K_X+B')-h^*(K_{Y}+B'_{Y})\ge 0$. Therefore $(X,B)$ has a log minimal model 
over $Z$ by Proposition \ref{p-WZD}.\\
\end{proof}

\begin{proof}(of Theorem \ref{t-flip-1} in general case)
Recall that we proved the theorem when $X$ is projective, in Section \ref{s-flips}.
By perturbing the coefficients, we can assume that $(X,B)$ is klt.
By Theorem \ref{t-mmodel}, $(X,B)$ has a log minimal model over $Z$, say $(X^+,B^+)$. 
Since $(X,B)$ is klt, $X\bir X^+$ is an isomorphism in codimension one. 
Let $H^+$ be an ample$/Z$ divisor on $X^+$ and let $H$ be its birational transform 
on $X$. Since $X\to Z$ is a $K_X+B$-negative extremal contraction, $K_X+B\equiv hH/Z$ 
for some $h>0$. Thus $K_{X^+}+B^+\equiv hH^+/Z$ which means that  
$K_{X^+}+B^+$ is ample$/Z$ so we are done.\\
\end{proof}

\section{The connectedness principle with applications to semi-ampleness}\label{s-connectedness}

\subsection{Connectedness}
In this subsection, we prove the connectedness principle in dimension $\le 3$. 
The proof is based on LMMP rather than vanishing theorems. 

The following lemma is essentially [\ref{Xu}, Proposition 2.3]. We recall its proof for convenience.

\begin{lem}\label{l-ample-dlt}
Let $(X,B)$ be a projective pair of dimension $\le 3$ over $k$. 
Assume that $(X,B)$ is klt (resp. dlt) and that $A$ is a nef and big (resp. ample) $\R$-divisor.
 Then there is $0\le A'\sim_\R A$ such that $(X,B+A')$ is klt (resp. dlt).
\end{lem}
\begin{proof}
First we deal with the dlt case.
Let $f\colon W\to X$ be a log resolution of $(X,B)$ which extracts only prime divisors with positive 
log discrepancy with respect to $(X,B)$. 
This exists by the definition of dlt pairs. The resolution is 
obtained by a sequence of blow ups with smooth centers, hence there is an $\R$-divisor 
$E'$ exceptional$/X$ with sufficiently small coefficients such that $-E'$ is ample$/X$ and 
$\Supp E'$ is the union of all the prime exceptional$/X$ divisors on $W$. Note that by the 
negativity lemma (\ref{ss-negativity}), $E'\ge 0$. Moreover, $f^*A-E'$ is ample$/X$.

Let $B_W$ be given by 
$$
K_W+B_W= f^*(K_X+B)
$$
By assumption, $B_W$ has coefficients at most $1$ and the coefficient of any 
prime exceptional$/X$ divisor is less than $1$. 
Let $A_W'\sim_\R f^*A-E'$ be general and let $A':=f_*A_W'$. Then $A'\sim_\R A$ and we can write 
$$
K_W+B_W+A_W'+E'= f^*(K_X+B+A')
$$
where we can make sure that the coefficients of $B_W+A_W'+E'$ are at most $1$ and that the coefficient of 
any prime exceptional$/X$ divisor is less than $1$ because the 
coefficients of $E'$ are sufficiently small. This implies that $(X,B+A')$ is dlt.

Now we deal with the klt case. Since $A$ is nef and big, by definition, 
$A\sim_\R G+D$ with $G\ge 0$ ample and $D\ge 0$. So by 
replacing $A$ with $(1-\epsilon)A+\epsilon G$ and replacing $B$ with $B+\epsilon D$ 
we can assume that $A$ is ample. Now apply the dlt case.\\
\end{proof}

\begin{proof}(of Theorem \ref{t-connectedness-d-3})
Assume that the statement does not hold for some $z$. 
By Lemma \ref{l-dlt-model-proj}, there is a $\Q$-factorial dlt pair $(Y,B_Y)$ and a 
birational morphism $g\colon Y\to X$ with $K_Y+B_Y$ nef$/X$, every exceptional divisor of 
$g$ is a component of $\rddown{B_Y}$, and $g_*B_Y=B$. 
Moreover, $K_Y+B_Y+E_Y=f^*(K_X+B)$ for some $E_Y\ge 0$ with $\Supp E_Y\subseteq \rddown{B_Y}$. 
Also the non-klt locus of $(Y,B_Y)$, that is $\rddown{B_Y}$, maps surjectively onto the non-klt locus of $(X,B)$ 
hence  $\rddown{B_Y}$ is not connected in some neighborhood of $Y_z$.

Now by assumptions, $K_Y+B_Y+E_Y+L_Y\sim_\R 0/Z$ for some globally nef and big $\R$-divisor $L_Y$.
Since $X$ is $\Q$-factorial, we can write $L_Y\sim_\R A_Y+D_Y$ where $A_Y$ is ample and $D_Y\ge 0$ 
is exceptional$/X$. In particular, $\Supp D_Y\subset \rddown{B_Y}$. By picking 
a general 
$$
G_Y\sim_\R \epsilon A_Y+(1-\epsilon)L_Y-\delta \rddown{B_Y}
$$ 
for some small $\delta>0$ 
and applying Lemma \ref{l-ample-dlt} we can assume that  $(Y,B_Y+G_Y)$ 
is dlt. By construction, 
$$
K_Y+B_Y+G_Y\sim_\R P_Y:=-\epsilon D_Y-E_Y-\delta\rddown{B_Y}/Z
$$
and $\Supp P_Y=\rddown{B_Y}$. 

Run a generalized LMMP$/Z$ on $K_Y+B_Y+G_Y$ as in the proof of Theorem \ref{t-mmodel}. 
We show that this is actually a usual LMMP hence it terminates by special termination (\ref{p-st}). 
 Assume that we have arrived at a model $Y'$ 
and let $R$ be a $K_{Y'}+B_{Y'}+G_{Y'}$-negative extremal ray$/Z$. Since $Y'\to Z$ is birational, 
$R$ is either a divisorial extremal ray or a flipping extremal ray. 
In the former case $R$ can be contracted by a projective morphism by Lemma \ref{l-div-ray}. 
So assume $R$ is a flipping extremal ray. Then the generalized flip $Y'\bir Y''/V$ exists by 
Theorem \ref{t-flip-2} where $Y'\to V$ is the contraction of $R$ to the algebraic space $V$.
Since $P_{Y'}\cdot R<0$, some component $S_{Y'}$ of $\rddown{B_{Y'}}$ intersects $R$ positively.  
Now there is a boundary $\Delta_{Y'}$ such that $(Y',\Delta_{Y'})$ is plt, $S_{Y'}=\rddown{\Delta_{Y'}}$, 
and $(K_{Y'}+\Delta_{Y'})\cdot R=0$. But then we can find $N_{Y''}\ge 0$ such that 
$(Y'',\Delta_{Y''}+N_{Y''})$ is plt and $(K_{Y''}+\Delta_{Y''}+N_{Y''})\cdot R<0$. 
Therefore by \ref{ss-pl-ext-rays} and \ref{ss-good-exc-locus}, $Y''\to V$ is a projective morphism which implies that 
$Y'\to V$ is also a projective morphism and that the flip is a usual flip.

We claim that the connected components of $\rddown{B_Y}$ over $z$ remain disjoint over $z$ 
in the course of the LMMP: assume not and let $Y'$
be the first model in the process such that there are irreducible components 
$S_Y,T_Y$  of $\rddown{B_Y}$  belonging to disjoint connected components over $z$ 
such that $S_{Y'},T_{Y'}$ intersect over $z$. 
Let $\Delta_Y=B_Y-\tau (\rddown{B_Y}-S_Y-T_Y)$ for some small $\tau>0$. Then $(Y,\Delta_Y+G_Y)$ 
is plt in some neighborhood of $Y_z$ because $\rddown{\Delta_Y+G_Y}=S_Y+T_Y$ and $S_Y,T_Y$ 
are disjoint over $z$. Moreover, $Y\bir Y'$ is a partial LMMP 
on $K_Y+\Delta_Y+G_Y$ hence $(Y',\Delta_{Y'}+G_{Y'})$ is also plt over $z$. But since $S_{Y'},T_{Y'}$ 
intersect over $z$, $(Y',\Delta_{Y'}+G_{Y'})$ cannot be plt over $z$, 
a contradiction.

Next we claim that no connected component of $\rddown{B_Y}$ over $z$ 
can be contracted by the LMMP (although some of their 
irreducible components might be contracted). By construction $-P_Y\ge 0$ and $\Supp -P_Y=\rddown{B_Y}$, and  
$-P_Y$ is positive on each extremal ray in the LMMP. Write $-P_Y=\sum -P_Y^i$ where $-P_Y^i$ 
are the connected components of $-P_Y$ over $z$. By the previous paragraph, $-P_Y^i$ and $-P_Y^j$ 
remain disjoint during the LMMP if $i\neq j$. Moreover, if we arrive a model $Y'$ in the LMMP 
on which we contract an extremal ray $R$, then $-P_{Y'}^j\cdot R>0$ for some $j$ and 
$-P_{Y'}^i\cdot R=0$ for $i\neq j$. Therefore the contraction 
of $R$ cannot contract any of the $-P_{Y'}^i$.

The LMMP ends up with a log minimal model $(Y',B_{Y'}+G_{Y'})$ over $Z$. Then 
$P_{Y'}$ is nef$/Z$.  Assume that $Y_z'\nsubseteq \Supp P_{Y'}$ set-theoretically. Since  $Y_z'$ intersects 
$\Supp P_{Y'}$, there is some curve $C\subset Y_z'$ not contained in $\Supp P_{Y'}$
but intersects it. Then as $-P_{Y'}\ge 0$ we have $-P_{Y'}\cdot C>0$ hence $P_{Y'}\cdot C<0$, 
 a contradiction. 
Now since $Y_z'$ is connected, it is contained in exactly one connected 
component of $\rddown{B_{Y'}}$ over $z$. This is a contradiction because by assumptions 
at least two connected components of $\rddown{B_{Y'}}$ over $z$ intersect the fibre $Y_z'$.\\
\end{proof}

We now show that a strong form of the connectedness principle holds on surfaces.

\begin{thm}\label{t-connectedness-d-2}
Let $(X,B)$ be a $\Q$-factorial projective pair of dimension $2$ over $k$. 
Let $f\colon X\to Z$ be a projective contraction (not necessarily birational) such that 
$-(K_X+B)$ is ample$/Z$. Then for any closed point $z\in Z$, the non-klt locus $N$ of 
$(X,B)$ is connected in any neighborhood of the fibre $X_z$ over $z$. More strongly, $N\cap X_z$ 
is connected.
\end{thm}
\begin{proof}
It is enough to prove the last claim.
Assume that  $N\cap X_z$ is not connected for some $z$. 
We use the notation and the arguments of the proof of Theorem \ref{t-connectedness-d-3}. 
Let $(Y,B_Y)$ be the pair constructed over $X$ and $Y\bir Y'$ the LMMP$/Z$  
on $K_Y+B_Y+G_Y\sim_\R P_Y$ and $h\colon Y'\to Z$ the corresponding map. The same arguments 
of the proof of Theorem \ref{t-connectedness-d-3} show that the connected components of 
$P_Y$ over $z$ remain disjoint in the course of the LMMP and none of them will be contracted.

 By assumptions, 
$\rddown{B_Y}\cap Y_z$ is not connected. 
We claim that the same holds in the course of the LMMP.
If not, then at some step of the LMMP we arrive at a model $W$ with a $K_W+B_W+G_W$-negative 
extremal birational contraction $\phi\colon W\to V$ such that $\rddown{B_W}\cap W_z$ is not connected but 
$\rddown{B_V}\cap V_z$ is connected. Let $C$ be the exceptional curve of $W\to V$. 
Now $\phi(\rddown{B_W})=\rddown{B_V}$: the inclusion $\supseteq$ is clear; 
the inclusion $\subseteq$ follows from the fact that if $C$ is a component of $\rddown{B_W}$, 
then at least one other irreducible component of $\rddown{B_W}$ intersects $C$ because $P_W\cdot C<0$. 
Therefore $\phi(\rddown{B_W}\cap W_z)=\rddown{B_V}\cap V_z$. Since $\rddown{B_V}\cap V_z$ is connected 
but $\rddown{B_W}\cap W_z$ is not connected, there exist two connected components of 
$\rddown{B_W}\cap W_z$ whose images under $\phi$ intersect. So there are closed points 
$w,w'$ belonging to different connected components of $\rddown{B_W}\cap W_z$ such that $\phi(w)=\phi(w')$. 
In particular, $w,w'\in C$. Note that $C$ is not a component of $\rddown{B_W}$ otherwise 
$C\subset \rddown{B_W}\cap W_z$ connects $w,w'$ which contradicts the assumptions.
Therefore $\rddown{B_W}\cap C$ is a finite set of closed points with more than one element.
Now perturbing the coefficients of 
$B_W$ we can find a $\Gamma_W\le \rddown{B_W}$ such that
$(W,\Gamma_W)$ is plt in a neighborhood of $C$, $(K_W+\Gamma_W)\cdot C<0$ and such 
that $\rddown{\Gamma_W}\cap C$ is a finite set of closed points with more than one element.
Then in a formal neighborhood of $\phi(w)$, $\rddown{\Gamma_V}$ has at least two branches which implies that 
$\rddown{\Gamma_V}$ is not normal which in turn contradicts the plt property of 
$(V,\Gamma_V)$. 

Since $\rddown{B_{Y'}}\cap Y_z'$ is not connected, there is a component $D$ of $Y_z'$ 
not contained in $\Supp P_{Y'}=\rddown{B_{Y'}}$ but intersects it. Thus $P_{Y'}$ cannot be nef$/Z$ 
as $-P_{Y'}\ge 0$. Therefore 
the LMMP terminates with a Mori fibre space $Y'\to Z'/Z$.  
If $Z'$ is a point, then  $\rddown{B_{Y'}}$ has at least two disjoint 
irreducible components which contradicts the fact that the Picard number $\rho(Y')=1$ in this case.
So we can assume that $Z'$ is a curve. 

Assume that $Z$ is also a curve in which case $Z'=Z$.
Let $F$ be the reduced variety associated to a general fibre of $Y'\to Z'$. 
Then by the adjunction formula we get $F\simeq \PP^1$, $K_{Y'}\cdot F=-2$, and 
$(B_{Y'}+G_{Y'})\cdot F<2$. On the other hand, 
since $\rddown{B_{Y'}}\cap Y'_z$ has at least two points, $\rddown{B_{Y'}}\cap F$ also 
has at least two points hence  
$$
(B_{Y'}+G_{Y'})\cdot F\ge (\rddown{B_{Y'}}+G_{Y'})\cdot F>2
$$  
which is a contradiction.
Now assume that $Z$ is a point. Since $\rddown{B_{Y'}}\cap Y_z'$ is not connected,
 $\rddown{B_{Y'}}$ has at least two disjoint 
connected components, say $M_{Y'},N_{Y'}$. On the other hand, since $P_{Y'}\cdot F<0$, 
we may assume that $M_{Y'}$ 
intersects $F$ (hence $M_{Y'}$ intersects every fibre of $Y'\to Z'$). 
If some component of $N_{Y'}$ is vertical$/Z'$, 
then $M_{Y'},N_{Y'}$ intersect a contradiction. Thus each component of $N_{Y'}$ is horizontal$/Z'$ 
hence they intersect each 
fibre of $Y'\to Z'$. But then we can get a contradiction as in the $Z'=Z$ case.\\
\end{proof}

\subsection{Semi-ampleness}
We use the connectedness principle on surfaces to prove some semi-ampleness results in dimension $2$ and $3$.
These are not only interesting on their own but also useful for the proof of the finite generation (\ref{t-fg}).

\begin{proof}(of Theorem \ref{t-sa-reduced-boundary})
Let $S\le \rddown{B}$ be a reduced divisor. Assume that 
$(K_X+B+A)|_S$ is not semi-ample. We will derive a contradiction. 
We can assume that if $S'\lneq S$ is any other reduced divisor, then 
$(K_X+B+A)|_{S'}$ is semi-ample. Note that $S$ cannot be irreducible by abundance for surfaces (cf. 
[\ref{Tanaka}]). Using the ample divisor $A$ and applying Lemma \ref{l-ample-dlt}, we can perturb the coefficients 
of $B$ so that we can assume $S=\rddown{B}$. 

Let $T$ be an irreducible component of $S$ and let $S'=S-T$. 
By assumptions,  $(K_X+B+A)|_{T}$ and $(K_X+B+A)|_{S'}$ are both semi-ample.
Let $g\colon T\to Z$ be the projective contraction associated to $(K_X+B+A)|_{T}$.
By adjunction define $K_T+B_T:=(K_X+B)|_T$ and $A_T=A|_T$. Since $K_T+B_T+A_T\sim_\Q 0/Z$ and 
since $A_T$ is ample, $-(K_T+B_T)$ is ample$/Z$. Moreover,  $S'\cap T=\rddown{B_T}$ as topological spaces. 
By the connectedness principle for surfaces (\ref{t-connectedness-d-2}), $\rddown{B_T}\to Z$ has 
connected fibres hence $S'\cap T\to Z$ also has connected fibres.
Now apply Keel [\ref{Keel}, Corollary 2.9].\\ 
\end{proof}

\begin{thm}\label{t-sa-reduced-boundary-2}
Let $(X,B+A)$ be a projective $\Q$-factorial dlt pair of dimension $3$ over $k$ of 
char $p>5$. Assume that 

$\bullet$ $A,B\ge 0$ are $\Q$-divisors with $A$ ample,

$\bullet$ $(Y,B_Y+A_Y)$ is a $\Q$-factorial weak lc model of $(X,B+A)$,

$\bullet$ $Y\bir X$ does not contract any divisor,

$\bullet$  $\Supp A_{Y}$ does not contain any lc centre of $(Y,B_{Y}+A_{Y})$, 

$\bullet$ if $\Sigma$ is a connected component of $\mathbb{E}(K_{Y}+B_{Y}+A_{Y})$ and 
$\Sigma \nsubseteq \rddown{B_{Y}}$, then 
$(K_{Y}+B_{Y}+A_{Y})|_{\Sigma}$ is semi-ample.\\

Then $K_{Y}+B_{Y}+A_{Y}$ is semi-ample. 
\end{thm}
\begin{proof}
Note that if $K_X+B+A$ is not big, then $\mathbb{E}(K_{Y}+B_{Y}+A_{Y})=Y$ 
hence the statement is trivial. So we can assume that $K_X+B+A$ is big.
Let $\phi$ denote the map $X\bir Y$ and let $U$ be the largest open set 
over which $\phi$ is an isomorphism. Then since $A$ is ample and $X$ is $\Q$-factorial, $\Supp A_Y$  
contains $Y\setminus \phi(U)$: indeed let $y\in Y\setminus \phi(U)$ be a closed point and let $W$ be the normalization 
of the graph of $\phi$, and $\alpha\colon W\to X$ and $\beta\colon W\to Y$ be the 
corresponding morphisms; first assume that $\dim \beta^{-1}\{y\}>0$; then 
$\alpha^*A$ intersects $\beta^{-1}\{y\}$ because $A$ is ample hence $\Supp A_Y$ contains 
$y$; now assume that $\dim \beta^{-1}\{y\}=0$; then $\beta$ is an isomorphism over 
$y$; on the other hand, $\alpha$ cannot be an isomorphism near $\beta^{-1}\{y\}$ otherwise 
$\phi$ would be an isomorphism near $\alpha(\beta^{-1}\{y\})$ hence $y\in \phi(U)$, 
a contradiction; thus as $X$ is $\Q$-factorial, $\alpha$ contracts some prime divisor 
$E$ containing $\beta^{-1}\{y\}$; but then $Y\bir X$ contracts a divisor, a contradiction. 

Let $C\ge 0$ be any $\Q$-divisor 
such that $(X,B+A+C)$ is dlt.  Then $(Y,B_{Y}+A_{Y}+\epsilon C_Y)$ is dlt 
for any sufficiently small $\epsilon>0$ because $(Y,B_{Y}+A_{Y})$ has no lc 
centre inside $Y\setminus \phi(U)\subset \Supp A_Y$. 
 Now let $G_{Y}\ge 0$ be a general small ample $\Q$-divisor 
on $Y$ and $G$ its birational transform on $X$. Since $G$ is small, $A-G$ is ample. 
Let $C\sim_\Q A-G$ be a general $\Q$-divisor. 
Let
$$
\Gamma_{Y}:=B_{Y}+(1-\epsilon)A_{Y}+\epsilon C_{Y}+\epsilon G_{Y}
$$
Then 
$$
K_{Y}+\Gamma_{Y}\sim_\Q  K_{Y}+B_{Y}+A_{Y}
$$
and $\rddown{B_{Y}}=\rddown{\Gamma_{Y}}$.
Moreover, by the above remarks and by Lemma \ref{l-ample-dlt} we can assume that 
$(Y,\Gamma_{Y})$ is dlt.

Now by Theorem \ref{t-sa-reduced-boundary}, $(K_{Y}+\Gamma_{Y})|_{\rddown{\Gamma_Y}}$ is semi-ample 
hence $(K_{Y}+B_{Y}+A_{Y})|_{\rddown{B_Y}}$ is semi-ample. Therefore $(K_{Y}+B_{Y}+A_{Y})|_{\Sigma}$ 
is semi-ample for any connected component of $\mathbb{E}(K_{Y}+B_{Y}+A_{Y})$ hence we can  
apply Theorem \ref{t-Keel-1}.\\
\end{proof}

\section{Finite generation and base point freeness}\label{s-fg}

\subsection{Finite generation} 
In this subsection we prove Theorem \ref{t-fg}.

\begin{lem}\label{l-fg-decrease}
Let $(X,B)$ be a pair  
and $M$ a $\Q$-divisor satisfying the following properties:

$(1)$ $(X,\Supp(B+M))$ is projective log smooth of dimension $3$ over $k$ of char $p>5$, 

$(2)$ $K_X+B$ is a big $\Q$-divisor,

$(3)$ $K_X+B\sim_\Q M\ge 0$ and $\rddown{B}\subset \Supp M\subseteq \Supp B$,

$(4)$ $M=A+D$ where $A$ is an ample $\Q$-divisor and $D\ge 0$,
 
$(5)$ $\alpha M=N+C$ for some rational number $\alpha>0$ such that $N,C\ge 0$ are $\Q$-divisors, 
$\Supp N=\rddown{B}$, and $(X,B+C)$ is dlt,

$(6)$ there is an ample $\Q$-divisor $ A'\ge 0$ such that $A'\le A$ and $A'\le C$.\\
  
If $(X,B+tC)$ has an lc model for some real number $t\in (0,1]$, then 
$(X,B+(t-\epsilon)C)$ also has an lc model for any 
sufficiently small $\epsilon>0$.
\end{lem}
\begin{proof}
We can assume that $C\neq 0$.
If we let $\Delta=B-\delta(N+C)$ for some small rational number $\delta>0$, 
then $(X,\Delta)$ is klt and $K_X+B$ is a positive multiple of $K_X+\Delta$ up to 
$\Q$-linear equivalence. 
Similarly, for any $s\in (0,1]$, there is $s'\in (0,s)$ such that $(X,\Delta+s'C)$ is klt 
and $K_X+B+sC$ is a positive multiple of $K_X+\Delta+s'C$ up to 
$\Q$-linear equivalence.  So if $(Y,\Delta_Y+s'C_Y)$ is a log minimal model of 
$(X,\Delta+s'C)$, which exists by Theorem \ref{t-mmodel}, then $(Y,B_Y+sC_Y)$ is a $\Q$-factorial 
weak lc model of $(X,B+sC)$ such that $Y\bir X$ 
does not contract divisors and $X\bir Y$ is $K_X+B+sC$-negative (see \ref{ss-divisors} for 
this notion). 
We will make use of this observation below.
 
Let $T$ be the lc model of $(X,B+tC)$ and let $(Y,B_Y+tC_Y)$ be a $\Q$-factorial weak lc model of $(X,B+tC)$ 
such that $X\bir Y$ is $K_X+B+tC$-negative and its inverse does not contract divisors. 
Then the induced map $Y\bir T$ is a morphism and $K_T+B_T+tC_T$ pulls back to 
$K_Y+B_Y+tC_Y$. 

First assume that $t$ is irrational. Then $C_Y\equiv 0/T$. Moreover, 
$C_T$ is $\Q$-Cartier because the set of those $s\in\R$ such that $K_T+B_T+sC_T$ is $\R$-Cartier 
forms a rational affine subspace of $\R$ (this can be proved using simple 
linear algebra similar to \ref{ss-ext-rays-scaling}). Since $t$ belongs to this affine subspace and $t$ is 
not rational, the affine subspace is equal to $\R$ hence $K_T+B_T+sC_T$ is $\R$-Cartier for 
every $s$ which implies that $C_T$ is $\Q$-Cartier. Thus $C_Y\sim_\Q 0/T$ hence  
$K_T+B_T+(t-\epsilon) C_T$ pulls back to $K_Y+B_Y+(t-\epsilon)C_Y$ and the former is ample  
for every sufficiently small $\epsilon>0$. This means that $T$ is also the lc model 
of $(X,B+(t-\epsilon)C)$. 

From now on we assume that $t$ is rational. Replace $Y$ with a $\Q$-factorial weak lc model 
of $(Y,B_Y+(t-\epsilon)C_Y)$ over $T$ so that $X\bir Y$ is still $K_X+B+(t-\epsilon)C$-negative. 
 Since $K_T+B_T+tC_T$ is ample, 
by choosing $\epsilon$ to be small enough, we can assume that 
$K_Y+B_Y+(t-\epsilon)C_Y$ is nef globally, by \ref{ss-ext-rays-II}. Then $(Y,B_Y+(t-\epsilon)C_Y)$ 
is a weak lc model of $(X,B+(t-\epsilon)C)$ hence it is enough to show that 
$K_Y+B_Y+(t-\epsilon)C_Y$ is semi-ample. 
Perhaps after replacing $\epsilon$ with a smaller number we can assume that 
$K_Y+B_Y+(t-\epsilon')C_Y$ also nef globally for some $\epsilon'>\epsilon$ and that $t-\epsilon$ is rational. 

Let $Y\to V$ be the 
contraction to an algebraic space associated to $K_Y+B_Y+(t-\epsilon)C_Y$. 
Any curve contracted by $Y\to V$ is also contracted by $Y\to T$ because $K_Y+B_Y+tC_Y$ 
and $K_Y+B_Y+(t-\epsilon')C_Y$ are both nef and $\epsilon'>\epsilon$.  Thus we get an induced map $V\to T$. 
Moreover, there is a small contraction $Y'\to V$ from a $\Q$-factorial normal projective 
variety $Y'$: recall that $(Y,\Lambda_Y:=\Delta_Y+t'C_Y)$ is klt where $\Delta$ and $t'$ are 
as in the first paragraph; now $Y'$ can be obtained by taking a log resolution $W\to Y$, defining 
$\Lambda_W$ to be the birational transform of $\Lambda_V$ plus the reduced exceptional 
divisor of $W\to V$, running an LMMP$/V$ on $K_W+\Lambda_W$, using special termination 
and the fact that $K_W+\Lambda_W\equiv E/V$ for some $E\ge 0$ whose support is 
equal to the reduced exceptional divisor of $W\to V$, and applying the negativity lemma (\ref{ss-negativity}).
Since $K_Y+B_Y+(t-\epsilon)C_Y\equiv 0/V$, $K_{Y'}+B_{Y'}+(t-\epsilon)C_{Y'})$ is also nef and 
the former is semi-ample if and only if the latter is. So by replacing $Y$ with  
$Y'$, we can in addition assume that $Y\to V$ is a small contraction.

Let $\Sigma$ be a connected component of the exceptional set of $Y\to V$. Since 
$Y\to V$ is a small morphism, $\Sigma$ is one-dimensional. On the other hand, 
since 
$$
K_{Y}+B_Y+(t-\epsilon)C_Y\equiv 0/V
$$ 
and 
$$
K_{Y}+B_Y+tC_Y\equiv 0/V
$$ 
we get $C_Y\equiv 0/V$ hence $N_Y\equiv 0/V$.
Therefore either $\Sigma\subset \Supp N_Y$ or $\Sigma\cap \Supp N_Y=\emptyset$.
Moreover, if $\Sigma\cap \Supp N_Y=\emptyset$, then $(K_Y+B_Y+(t-\epsilon)C_Y)|_\Sigma$ 
is semi-ample because near $\Sigma$ the divisor  
$K_Y+B_Y+(t-\epsilon)C_Y$ is a multiple of $K_Y+B_Y+tC_Y$ 
and the latter is semi-ample. 

We can assume that $ A'$ in (6) has small coefficients. Let $B'=B+(t-\epsilon)C-A'$. Since $(Y,B_Y'+A_Y'+\epsilon C_Y)$ 
is lc, $\Supp C_Y$ (hence also $\Supp A'_Y$) does not contain any lc centre of $(Y,B_Y'+A_Y')$.
Now applying Theorem \ref{t-sa-reduced-boundary-2} to 
$(X,B'+A')$ shows that  
$K_Y+B_Y+(t-\epsilon)C_Y$ is semi-ample (note that the exceptional locus of $Y\to V$ is 
equal to $\mathbb{E}(K_Y+B_Y'+A_Y')$). Therefore, $K_Y+B_Y+sC_Y$ is semi-ample for every $s\in[t-\epsilon,t]$.\\
\end{proof}

\begin{prop}\label{p-fg-1-4}
Let $(X,B)$ be a pair and $M$ a $\Q$-divisor satisfying properties $(1)$ to $(4)$ 
of Lemma \ref{l-fg-decrease}. Then the lc ring $R(K_X+B)$ is finitely generated.
\end{prop}
\begin{proof}
\emph{Step 1.} 
We follow the proof of [\ref{B-mmodel}, Proposition 3.4], which is similar to 
[\ref{BCHM}, \S 5], but with some twists. 
Assume that $R(K_X+B)$ is not finitely generated. 
We will derive a contradiction. 
By replacing $A$ with $\frac{1}{m}S$ where $m$ is sufficiently divisible and $S$ is a 
general member of $|mA|$, and changing $M,B$ accordingly, we can assume that 

$(7)$ $S:=\Supp A$ is irreducible and $K_X+S+\Delta$ is ample for any boundary $\Delta$ 
supported on $\Supp(B)-S$.\\ 

Let ${\theta}(X,B,M)$ be the number of those components of $M$ which are not components 
of $\rddown{B}$ (such $\theta$ functions were defined in \ref{ss-WZD} in a more general setting). 
By (7),  $S$ is not a component of $\rddown{B}$, hence ${\theta}(X,B,M)>0$ 
otherwise $K_X+B$ is ample and $R(K_X+B)$ is finitely generated, 
a contradiction. 
Define
$$
\alpha:=\min\{t>0~|~~\rddown{(B+tM)^{\le 1}}\neq \rddown{B}~\}
$$
where for a divisor $R=\sum r_iR_i$ we define $R^{\le 1}=\sum r_i'R_i$ with $r_i'=\min\{r_i,1\}$.
 In particular, $(B+\alpha M)^{\le 1}=B+C$ for some $C\ge 0$ 
supported in $\Supp M$, and $\alpha M=C+N$ 
where $N\ge 0$ is supported in $\rddown{B}$ and $C$ has no common components with $\rddown{B}$. 

Property (3) ensures that $\Supp N=\rddown{B}$, and by property (7) we have $\alpha A\le C$. 
So $(X,B)$ and $M$ also satisfy properties (5) and (6) of \ref{l-fg-decrease} with $A'=\alpha' A$ 
for some $\alpha'>0$.

\emph{Step 2.}
Let $B':=B+C$ and let $M':=M+C$.  
Then the pair $(X,B')$ is log smooth dlt and   
$$
{\theta}(X,B',M')<{\theta}(X,B,M)
$$ 
Assume that $R(K_X+B')$ is not finitely generated. By (7), $S$ is not a 
component of $\rddown{B'}$ and ${\theta}(X,B',M')>0$.
Now replace $(X,B)$ with $(X,B')$, replace $D$ with $D':=D+C$, 
and replace $M$ with $M'$. By construction, all the properties (1) to (4) of \ref{l-fg-decrease} 
and property (7) above are still satisfied. 
Repeating the above process we get to the situation in which either 
$R(K_X+B')$ is finitely generated, or ${\theta}(X,B',M')=0$ and $K_X+B'$ is ample. 
 Thus in any case we can assume $R(K_X+B')$ is finitely generated.

\emph{Step 3.} Let
$$
\mathcal T=\{t\in [0,1]~|~ (X,B+tC)~~\mbox{has an lc model}\}
$$
Since $R(K_X+B'=K_X+B+C)$ is finitely generated, 
$1\in\mathcal T$ hence $\mathcal T\neq \emptyset$. Moreover, if $t\in\mathcal T\cap (0,1]$,  
then by Lemma \ref{l-fg-decrease}, $[t-\epsilon,t]\subset \mathcal{T}$ for some $\epsilon>0$.  
Now let $\tau=\inf \mathcal{T}$. If $\tau\in \mathcal{T}$, then $\tau=0$ which implies that 
$R(K_X+B)$ is finitely generated, a 
contradiction. So we may assume $\tau\notin \mathcal{T}$. There is a sequence 
$t_1>t_2>\cdots$ of rational numbers in $\mathcal{T}$ approaching $\tau$.
For each $i$, there is a $\Q$-factorial weak lc model $(Y_i,B_{Y_i}+t_iC_{Y_i})$ 
of $(X,B+t_iC)$ such that $Y_i\bir X$ does not 
contract divisors (see the beginning of the proof of Lemma \ref{l-fg-decrease}). 
By taking a subsequence, we can assume that 
all the $Y_i$ are isomorphic in codimension one. In particular, 
$N_\sigma(K_{Y_1}+B_{Y_1}+\tau C_{Y_1})=0$.

Arguing as in the proof of Theorem \ref{t-sa-reduced-boundary-2}, we can show that 
$({Y_1},B_{Y_1}+\tau C_{Y_1})$ is dlt because $\alpha A\le C$ is ample and $\Supp A_{Y_1}$ 
does not contain any lc centre of $({Y_1},B_{Y_1}+\tau C_{Y_1})$.
Run the LMMP on $K_{Y_1}+B_{Y_1}+\tau C_{Y_1}$ with scaling of $(t_1-\tau)C_{Y_1}$ as in \ref{ss-g-LMMP-scaling}. 
Since $\alpha M_{Y_1}=N_{Y_1}+C_{Y_1}$, the LMMP is also an LMMP on $N_{Y_1}$. 
Thus each extremal ray in the process is a pl-extremal ray hence they can 
be contracted by projective morphisms (\ref{ss-pl-ext-rays}). Moreover, 
the required flips exist by Theorem \ref{t-flip-1}, and the LMMP terminates 
with a model $Y$ on which $K_{Y}+B_{Y}+\tau C_{Y}$ is nef, by special termination (\ref{p-st}). 
Note that the LMMP does not contract any divisor by the $N_\sigma=0$ property. 
Moreover, $K_{Y}+B_{Y}+(\tau+\delta) C_{Y}$ is nef for some $\delta>0$. Now, by replacing the sequence 
we can assume that $K_{Y}+B_{Y}+t_i C_{Y}$ is nef for every $i$ and by replacing each $Y_i$ with $Y$  
we can assume that $Y_i=Y$ for every $i$.  
A simple comparison of discrepancies (cf. [\ref{B-mmodel}, Claim 3.5]) shows 
that $(Y,B_{Y}+\tau C_{Y})$ is a $\Q$-factorial 
weak lc model of $(X,B+\tau C)$.
 
\emph{Step 4.}  Let $T_i$ be the lc model of 
$(X,B+t_iC)$. Then the map $Y\bir T_i$ is a morphism and $K_{Y}+B_{Y}+t_i C_{Y}$ is the 
pullback of an ample divisor on $T_i$.
 Moreover, for each $i$, the map $T_{i+1}\bir T_i$ is a morphism because any curve 
 contracted by $Y\to T_{i+1}$ is also contracted by $Y\to T_i$. So perhaps after replacing the sequence, 
we can assume that $T_i$ is independent of $i$ so we can drop the subscript and simply 
use $T$. Since $C\sim_\Q 0/T$, we can replace $Y$ with a $\Q$-factorialization of $T$ 
so that we can assume that $Y\to T$ is a small morphism (such a $\Q$-factorialization exists 
by the observations in the first paragraph of the proof of Lemma \ref{l-fg-decrease}).

Assume that $\tau$ is irrational. 
If $K_Y+B_Y+(\tau-\epsilon)C_Y$ is nef for 
some $\epsilon>0$, then $K_Y+B_Y+\tau C_Y$ is semi-ample because in this case 
$K_T+B_T+(\tau-\epsilon)C_T$ is nef and $K_T+B_T+t_i C_T$ is ample hence 
$K_T+B_T+\tau C_T$ is ample. If there is no $\epsilon$ as above, then 
by \ref{ss-ext-rays-scaling} and \ref{ss-ext-rays-II}, there is a curve 
$\Gamma$ generating some extremal ray such that $(K_Y+B_Y+\tau C_Y)\cdot \Gamma=0$ 
and $C_Y\cdot \Gamma>0$. 
This is not possible since $\tau$ is assumed to be irrarional. 
 So from now on we assume that 
$\tau$ is rational.

\emph{Step 5.} Let $Y\to V$ be the contraction to an algebraic space associated to $K_Y+B_Y+\tau C_Y$. 
This map factors through $Y\to T$ so we get an induced map $T\to V$.  
We can write 
$$
K_T+B_T+\tau C_T=a(K_T+B_T+t_iC_T)+bN_T
$$
for some $i$ and some rational numbers $a,b>0$. Since $K_T+B_T+t_iC_T$ is ample, we get 
$$
\mathbb{E}(K_T+B_T+\tau C_T)\subset \Supp N_T=\rddown{B_T}
$$
Thus since $N_Y\sim_\Q 0\sim_\Q C_Y/T$, the locus 
$\mathbb{E}(K_Y+B_Y+\tau C_Y)$ is a subset of the union of $\Supp N_Y=\rddown{B_Y}$ and the exceptional 
set of $Y\to T$.
 Let $\Lambda$ be a connected component of the exceptional 
set of $Y\to T$. Then, since $N_Y\sim_\Q 0/T$ and since $\Lambda$ is one-dimensional, 
either $\Lambda\subset \Supp N_{Y}$ or $\Lambda\cap \Supp N_{Y}=\emptyset$.
Therefore if $\Sigma$ is a connected component of $\mathbb{E}(K_Y+B_Y+\tau C_Y)$, then 
either $\Sigma\subset \Supp N_{Y}$ or $\Sigma\cap \Supp N_{Y}=\emptyset$.
In the latter case,  $(K_Y+B_Y+\tau C_Y)|_\Sigma$ is semi-ample because near $\Sigma$ the divisor  
$K_Y+B_Y+\tau C_Y$ is a multiple of $K_Y+B_Y+t_i C_Y$ and the latter is semi-ample. 
Finally as in the end of the proof of Lemma \ref{l-fg-decrease} we can apply Theorem 
\ref{t-sa-reduced-boundary-2} to show that $K_Y+B_Y+\tau C_Y$ is semi-ample. 
This is a contradiction because we assumed $\tau\notin\mathcal{T}$.\\
\end{proof}

\begin{proof}(of Theorem \ref{t-fg})
First assume that $Z$ is a point.
Pick $M\ge 0$ such that $K_X+B\sim_\Q M$. We can choose $M$ so that 
$M=A+D$ where $A\ge 0$ is ample and $D\ge 0$.  
 Let $f\colon W\to X$ be a log resolution 
of $(X,\Supp (B+M))$. Since $(X,B)$ is klt, we can write 
$$
K_W+B_W=f^*(K_X+B)+E
$$
where $(W,B_W)$ is klt, $K_W+B_W$ is a $\Q$-divisor, and $E\ge 0$ is exceptional$/X$. 
 Moreover, there is $E'\ge 0$ exceptional$/X$ 
such that $-E'$ is ample$/X$ (cf. proof of Lemma \ref{l-ample-dlt}). Let $A_W\sim_\Q f^*A-E'$ 
be general and let $D_W=f^*D+E+E'$. Then 
$$
K_W+B_W\sim_\Q M_W:=A_W+D_W
$$ 
Now replace $(X,B)$ with $(W,B_W)$, 
replace $M$ with $M_W$, and replace $A$ and $D$ with $A_W$ and $D_W$. 
Moreover, by adding a small multiple of $M$ to $B$ we can also assume that 
$\Supp M\subseteq \Supp B$.
Then $(X,B)$ and $M$ satisfy the properties (1) to (4) of Lemma \ref{l-fg-decrease}. 
Therefore, by Proposition \ref{p-fg-1-4}, $R(K_X+B)$ is finitely generated. 

Now we treat the general case, that is, when $Z$ is not necessarily a point.
By taking projectivizations of $X,Z$ and taking a log resolution, we 
may assume that $X,Z$ are projective and that $(X,B)$ is log smooth. We can also assume that 
$K_X+B\sim_\Q M=A+D/Z$ where $A$ is an ample $\Q$-divisor and $D\ge 0$. By adding 
some multiple of $M$ to $B$ we may assume $\Supp M\subseteq \Supp B$. Let $(Y,B_Y)$ be a log minimal model 
of $(X,B)$ over $Z$. Let $H$ be the pullback of an 
ample divisor on $Z$. Since $A\le B$, for each integer $m\ge 0$, there is $\Delta$ such that 
$K_X+B+mH\sim_\Q K_X+\Delta$ is big globally and that $(X,\Delta)$ is klt. Moreover, $(Y,\Delta_Y)$ is a 
log minimal model of $(X,\Delta)$ over $Z$. Now 
by \ref{ss-ext-rays-II}, if $m\gg 0$, then $K_Y+\Delta_Y$ 
is big and globally nef. On the other hand, $R(K_Y+\Delta_Y)$ is finitely generated over $k$ 
which means that $K_Y+\Delta_Y$ is semi-ample. Therefore $K_Y+B_Y$ is semi-ample$/Z$ 
hence $\mathcal{R}(K_X+B/Z)$ is a finitely generated $\mathcal{O}_Z$-algebra.
 
\end{proof}

\subsection{Base point freeness} 
\begin{proof}(of Theorem \ref{t-bpf})
It is enough to show that $\mathcal{R}(D/Z)$ is a finitely generated $\mathcal{O}_Z$-algebra. 
By taking a $\Q$-factorialization using Theorem \ref{cor-dlt-model}, we may assume that 
$X$ is $\Q$-factorial.  
Let $A=D-(K_X+B)$ which is nef and big$/Z$ by assumptions. By replacing $A$, and replacing $B$ accordingly, 
we may assume that $A$ is ample globally. 
By Lemma \ref{l-ample-dlt}, we can change $A$ up to $\Q$-linear equivalence so that 
$(X,B+A)$ is klt. But then $\mathcal{R}(K_X+B+A/Z)$ is finitely generated 
by Theorem \ref{t-fg} hence $\mathcal{R}(D/Z)$ is also finitely generated.
\end{proof}

\subsection{Contractions}
\begin{proof}(of Theorem \ref{t-contraction}) 
We may assume that $B$ is a $\Q$-divisor and that $(X,B)$ is klt. We can assume $N=H+D$ where $H$ is ample$/Z$ 
and $D\ge 0$. Let $G$ be the pullback of an ample divior on $Z$, and let  
 $N'=mG+nN+\epsilon H+\epsilon D$  where $\epsilon>0$ is sufficiently small  
and $m\gg n\gg 0$. Then we can 
find $A\sim_\Q N'$ such that $(X,B+A)$ is klt, $K_X+B+A$ is globally big, and $(K_X+B+A)\cdot R<0$. 
 By \ref{ss-ext-rays-II}, we can find an ample divisor $E$ such that $L:=(K_X+B+A+E)$ 
 is nef and big globally and $L^\perp=R$. We can also assume that $(X,B+A+E)$ is klt 
 hence by Theorem \ref{t-bpf}, $L$ is semi-ample which implies that $R$ can be contracted by a 
projective morphism.\\
\end{proof}

\section{ACC for lc thresholds}\label{s-ACC}

In this section, we prove Theorem \ref{t-ACC} by a method similar to the characteristic $0$ case 
(see [\ref{Kollar+}, Chapter 18] and [\ref{MP}]).
Let us recall the definition of \emph{lc threshold}. Let $(X,B)$ be an lc pair over $k$ and 
$M\ge 0$ an $\R$-Cartier divisor. The lc threshold of $M$ with respect to $(X,B)$ is defined as 
$$
\lct(M,X,B)=\sup \{t\mid (X,B+tM)~~\mbox{is lc}\}
$$

We first prove some results, including ACC for lc thresholds, for surfaces 
before we move on to 3-folds.

\subsection{ACC for lc thresholds on surfaces}

\begin{prop}\label{p-acc-surfaces}
ACC for lc thresholds holds in dimension $2$ (formulated similar to \ref{t-ACC}).
\end{prop}
\begin{proof}
If this is not the case, then there is a sequence $(X_i,B_i)$ of lc pairs of 
dimension $2$ over $k$ and $\R$-Cartier divisors $M_i\ge 0$ such that the coefficients of 
$B_i$ are in $\Lambda$, the coefficients of $M_i$ are in $\Gamma$ but such that 
the $t_i:=\lct(M_i,X_i,B_i)$ form a strictly increasing sequence of numbers.
If for infinitely many $i$, $(X_i,\Delta_i:=B_i+t_iM_i)$ has an lc centre of dimension one 
contained in $\Supp M_i$, 
then it is quite easy to get a contradiction. 
We may then assume that each $(X_i,\Delta_i)$ has an lc centre $P_i$ of dimension zero 
contained in $\Supp M_i$. We may also assume that $(X_i,\Delta_i)$ is plt outside $P_i$. 
Let $(Y_i,\Delta_{Y_i})$ be a $\Q$-factorial dlt model of 
$(X_i,\Delta_i)$ such that there are some exceptional divisors on $Y_i$ 
mapping to $P_i$. Such $Y_i$ exist by a version of Lemma \ref{l-extraction} in dimension $2$. 

There is a prime exceptional divisor $E_i$ of $Y_i\to X_i$ such that it intersects the 
birational transform of $M_i$. Note that $E_i$ is normal and actually isomorphic to $\PP^1_k$ 
since $E_i$ is a component of $\rddown{\Delta_{Y_i}}$ and $(K_{Y_i}+\Delta_{Y_i})\cdot E_i=0$. Now by adjunction define 
$K_{E_i}+\Delta_{E_i}=(K_{Y_i}+\Delta_{Y_i})|_{E_i}$. Then by Proposition \ref{p-adjunction-DCC} and 
its proof, the set of all the coefficients of the 
$\Delta_{E_i}$ is a subset of a fixed DCC set but they do not satisfy ACC. 
This is a contradiction since $\deg \Delta_{E_i}=2$.
\end{proof}

We apply the ACC of \ref{p-acc-surfaces} to negativity of contractions.

\begin{lem}\label{l-lim-nefness}
Let $\Lambda\subset [0,1]$ be a DCC set of real numbers. Then there is $\epsilon>0$ satisfying 
the following: assume we have 

$\bullet$ a klt pair $(X,B)$ of dimension $2$,

$\bullet$ the coefficients of $B$ belong to $\Lambda\cup [1-\epsilon,1]$, 

$\bullet$  $f\colon X\to Y$ is an extremal birational projective contraction with exceptional divisor $E$, 

$\bullet$ the coefficient of $E$ in $B$ belongs to $[1-\epsilon,1]$, and 

$\bullet$ $-(K_X+B)$ is nef$/Y$.\\ 

If $\Delta$ is obtained from $B$ by replacing each coefficient in $[1-\epsilon, 1]$ with $1$, 
then  $-(K_X+\Delta)$ is also nef$/Y$. 
\end{lem}
\begin{proof}
Note that klt pairs of dimension $2$ are $\Q$-factorial so $K_X+\Delta$ is $\R$-Cartier.   
 By \ref{p-acc-surfaces}, 
we can pick $\epsilon>0$ so that: if $(T,C)$ is lc of dimension $2$ and $M\ge 0$ such that 
the coefficients of $C$ belong to $\Lambda$ and the coefficients of $M$ belong to $\{1\}$, 
then the lc threshold $\lct(M,T,C)$ does not belong to $[1-\epsilon,1)$.

 Now since $(X,B)$ is klt and $-(K_X+B)$ is nef$/Y$, 
$(Y,B_Y)$ is also klt. Thus $(Y,\Delta_Y-\epsilon\rddown{\Delta_Y})$ is klt because 
$B_Y\ge \Delta_Y-\epsilon\rddown{\Delta_Y}$. 
In particular, the lc threshold of $\rddown{\Delta_Y}$ with respect to $(Y,\Delta_Y-\rddown{\Delta_Y})$ is 
at least $1-\epsilon$. Note that the coefficients of $\Delta$ belong $\Lambda\cup \{1\}$ 
and the coefficients of $\Delta_Y-\rddown{\Delta_Y}$ belong to $\Lambda$.
Thus by our choice of $\epsilon$, the pair $(Y,\Delta_Y)$ is lc. Therefore we can write  
$$
K_X+\Delta=f^*(K_Y+\Delta_Y)+eE
$$
for some $e\ge 0$ because the coefficient of $E$ in $\Delta$ is $1$. 
This implies that $-(K_X+\Delta)$ is indeed nef$/Y$.\\
\end{proof}

\subsection{Global ACC for surfaces}
In this subsection we prove a global type of ACC for surfaces (\ref{p-ACC-global}) which 
will be used in the proof of Theorem \ref{t-ACC}.

\begin{constr}\label{rem-Fano-twist}
Let $\epsilon\in (0,1)$ and let $X'$ be a klt Fano surface with $\rho(X')=1$.
Assume that $X'$ is not $\epsilon$-lc.
Pick a prime divisor $E$ (on birational models of $X'$) with log discrepancy $a(E,X',0)<\epsilon$. 
By a version of Lemma \ref{l-extraction} 
in dimension two, there is a
birational contraction $Y'\to X'$ which is extremal and has $E$ as the only exceptional
divisor.  Under our assumptions  
it is easy to find a boundary $D_{Y'}$ such that $(Y',D_{Y'})$ is klt and $K_{Y'}+D_{Y'}\sim_\R -eE$ 
for some $e>0$. In particular, we can run an LMMP on $-E$ which ends with a Mori fibre space $X''\to T''$
so that $E''$ positively intersects the extremal ray defining $X''\to T''$ where $E''$ is the 
birational transform of $E$. 

As $\rho(X')=1$, we get $\rho(Y')=2$. One of the extremal rays of $Y'$ 
gives the contraction $Y'\to X'$. The other one either gives $X''\to T''$ with $Y'=X''$ 
or it gives a birational contraction $Y'\to X''$.
If $\dim T''=0$, then $X''$ is also a klt Fano with
$\rho(X'')=1$. 
\end{constr}

\begin{lem}\label{l-bnd-comps-surfaces}
Let $b\in (0,1)$ be a real number. Then there is a natural number $m$ depending only 
on $b$ such that: let $(X,B)$ be a klt pair of dimension $2$ and $x\in X$ a closed point; 
then the number of those components of 
$B$ containing $x$ and with coefficient $\ge b$ is at most $m$. 
\end{lem}
\begin{proof}
Since $(X,B)$ is klt and $\dim X=2$, $X$ is $\Q$-factorial. 
We can assume that each coefficient of $B$ is equal to $b$ by discarding any component 
with coefficient less than $b$ and by decreasing each coefficient which is more than $b$. 
Moreover, we can assume every component of $B$ contains $x$.

Pick a nonzero $\R$-Cartier divisor $G\ge 0$ such that $(X,C:=B+G)$ is lc  
near $x$ and such that $x$ is a lc centre of $(X,B+G)$: for example we can take a 
log resolution $W\to X$ and let $G$ be the pushdown of an appropriate 
ample $\R$-divisor on $W$. Shrinking $X$ we can assume $(X,C)$ is lc. 
Since $(X,B)$ is klt, there is an extremal contraction $f\colon Y\to X$ 
which extracts a prime divisor $S$ with log discrepancy $a(S,X,C)=0$. 

Let $B_Y$ be the sum of $S$ and the birational transform of $B$. 
Then $-(K_Y+B_Y)$ is ample$/X$. Apply adjunction (\ref{p-adjunction-DCC}) and write 
$K_{S^\nu}+B_{S^\nu}$ for the pullback of $K_Y+B_Y$ to the normalization of $S$. 
As $-(K_{S^\nu}+B_{S^\nu})$ is ample, ${S^\nu}\simeq \PP^1$ and 
$\deg B_{S^\nu}<2$. 

By \ref{p-adjunction-DCC}, the coefficient of each $s\in \Supp B_{S^\nu}$ 
is of the form $\frac{n-1}{n}+\frac{rb}{n}$ for some integer $r\ge 0$ and some $n\in \N\cup \{\infty\}$. 
In particular, 
the number of the components of $B_{S^\nu}$ is bounded and the number $r$ in the 
formula is also bounded. This bounds the number of the components of $B$ because $r$ is 
more than or equal to the number of those  
components of $B_Y-S$ which pass through the image of $s$.\\  
\end{proof}

\begin{prop}\label{p-ACC-global}
Let $\Lambda\subset [0,1]$ be a DCC set of real numbers. Then there is a 
finite subset $\Gamma\subset \Lambda$ with the following property: 
let $(X,B)$ be a pair and $X\to Z$ a projective morphism such that 

$\bullet$ $(X,B)$ is lc of dimension $2$ over $k$,

$\bullet$ the coefficients of $B$ are in $\Lambda$, 

$\bullet$ $K_X+B\equiv 0/Z$,

$\bullet$ $\dim X>\dim Z$.

Then the coefficient of each horizontal$/Z$ component of $B$ is in $\Gamma$. 
\end{prop}
\begin{proof}
\emph{Step 1.}
We can assume that $1\in \Lambda$.
If the proposition is not true, then there is a sequence $(X_i,B_i), X_i\to Z_i$ 
of pairs and morphisms as in the proposition such that the set of the coefficients 
of the horizontal$/Z_i$ components of all the $B_i$ put together does not 
satisfy ACC. By taking $\Q$-factorial dlt 
models we can assume that $(X_i,B_i)$ are $\Q$-factorial dlt. 
Write $B_i=\sum b_{i,j}B_{i,j}$. We may assume that $B_{i,1}$ is horizontal$/Z_i$ 
and that $b_{1,1}<b_{2,1}<\cdots$. 

\emph{Step 2.}
First assume that $\dim Z_i=1$ for every $i$. Run the LMMP$/Z_i$ on 
$K_{X_i}+B_i-b_{i,1}B_{i,1}$ with scaling of $b_{i,1}B_{i,1}$. 
This terminates with a model $X_i'$ having an extremal contraction 
$X_i'\to Z_i'/Z_i$ such that $K_{X_i'}+B_i'-b_{i,1}B_{i,1}'$ is numerically negative over $Z_i'$. 
Let $F_i'$ be the reduced variety associated to a general fibre of $X_i'\to Z_i'$. 
Since $K_{X_i'}+B_i'\equiv 0/Z'$ and ${F'_i}^2=0$, we get $(K_{X_i'}+B_i'+F_i')\cdot F_i'=0$ 
hence the arithmetic genus $p_a(F_i')<0$ which implies that $F_i'\simeq \PP^1_k$. 
 We can write 
$$
\deg (K_{X_i'}+B_i'+F_i')|_{F_i'}=-2+\sum n_{i,j}b_{i,j}= 0
$$
for certain integers $n_{i,j}\ge 0$ such that $n_{i,1}>0$. Since the $b_{i,j}$ 
belong to the DCC set $\Lambda$, $n_{i,1}$ is bounded from above and below. 
Moreover, we can assume that the sums $\sum_{j\ge 2} n_{i,j}b_{i,j}$ satisfy the DCC hence 
$n_{i,1}b_{i,1}=2-\sum_{j\ge 2} n_{i,j}b_{i,j}$ satisfies the ACC, a contradiction. 

\emph{Step 3.}
From now on we may assume that $\dim Z_i=0$ for every $i$. Run the 
LMMP$/Z_i$ on $K_{X_i}+B_i-b_{i,1}B_{i,1}$ with scaling of $b_{i,1}B_{i,1}$. 
This terminates with a model $X_i'$ having an extremal contraction 
$X_i'\to Z_i'$ such that $K_{X_i'}+B_i'-b_{i,1}B_{i,1}'$ is numerically negative over $Z_i'$. 
If $\dim Z_i'=1$ for infinitely many $i$, then we get a contradiction by Step 2. 
So we assume that $Z_i'$ are all points hence each $X_i'$ is a Fano with Picard number one.  

Assume that $({X_i}',B_i')$ is lc but not klt for every $i$. Assume that each $({X_i}',B_i')$ 
has an lc centre $S_i'$ of dimension one. Let $K_{S_i'}+B_{S_i'}=(K_{X_i'}+B_i')|_{S_i'}$ 
by adjunction. Note that $S_i'$ is normal since $({X_i'},B_i'-b_{i,1}B_{i,1}')$ 
is $\Q$-factorial dlt. Since $K_{S_i'}+B_{S_i'}\equiv 0$, $S_i'\simeq \PP^1_k$.
If $\Supp B_{i,1}'$ contains an lc centre for infinitely many $i$, then we get a contradiction 
by ACC for lc thresholds in dimension $2$. So we can assume that $\Supp B_{i,1}'$ does not 
contain any lc centre, in particular, none of the points of $S_i'\cap B_{i,1}'$ is 
an lc centre. Now, since $\{b_{i,j}\}$ does not satisfy ACC, by Proposition 
\ref{p-adjunction-DCC}, the set of the coefficients of all the $B_{S_i'}$ 
satisfies DCC but not ACC which gives a contradiction as above (by considering the coefficients of 
the points in $S_i'\cap B_{i,1}'$). So we can assume that each $({X_i}',B_i')$ 
has an lc centre of dimension zero. 
By a version of Lemma \ref{l-extraction} in dimension $2$, there is a projective birational contraction $Y_i'\to X_i'$ 
which extracts only one prime divisor $E_i'$ and it satisfies $a(E_i',X_i',B_i')=0$.
Let $K_{Y_i'}+B_{Y_i'}$ be the pullback of $K_{X_i'}+B_i'$. By running the 
LMMP on $K_{Y_i'}+B_{Y_i'}-E_i'$, we arrive on a model on which either 
the birational transform of $E_i'$ intersects the birational transform of $B_{i,1}'$ for infinitely many $i$, 
or we get a Mori fibre space over a curve whose general fibre intersects the birational 
transform of $B_{i,1}'$ for infinitely many $i$. 
In any case, we can apply the arguments above to get a contradiction. 
So from now on we may assume that $({X_i}',B_i')$ are all klt.

\emph{Step 4.}
If there is $\epsilon>0$ such that $X_i'$ is $\epsilon$-lc for every $i$, then 
we are done since such $X_i'$ are bounded by Alexeev [\ref{Alexeev}]. So we can assume 
that the minimal log discrepancies of the $X_i'$ form a strictly decreasing 
sequence of positive numbers. Since $({X_i}',B_i')$ are klt, we can assume that 
the minimal log discrepancies of the $(X_i',B_i')$ also form a strictly decreasing 
sequence of positive numbers. As in Construction \ref{rem-Fano-twist}, we find a 
contraction ${Y}_i'\to X_i'$ extracting a prime divisor $E_i$ with minimal log discrepancy  
$a(E_i,X_i',B_i')<\epsilon$ and run a $-E_i$-LMMP to get a Mori fibre structure $X_i''\to Z_i''$. 
If $\dim Z_i''=1$ for each $i$, we use Step 2 to get a contradiction. So we may assume 
that $\dim Z_i''=0$ for each $i$. Note that the exceptional divisor of $X_i''\bir X_i'$ 
is a component of $B_i''$ with coefficient $\ge 1-\epsilon$ where $K_{X_i''}+B_i''$ 
is the pullback of $K_{X_i'}+B_i'$. 

Write $K_{Y_i'}+{B}_{Y_i'}$ for the pullback of $K_{X_i'}+B_i'$. 
By construction, the coefficients of ${B}_{Y_i'}$ belong to some DCC subset of 
$\Lambda\cup [1-\epsilon,1]$. We show that if $\epsilon$ is sufficiently small, then 
${Y}_i'\to X_i''$ cannot contract a component of ${B}_{Y_i'}$ with coefficient $\ge 1-\epsilon$.
Indeed let $\Delta_{Y_i'}$ be obtained from  ${B}_{Y_i'}$ by replacing each 
coefficient $\ge 1-\epsilon$ with $1$. Then by Lemma \ref{l-lim-nefness}, 
$-(K_{Y_i'}+{\Delta}_{Y_i'})$ is nef over both $X_i'$ and $X_i''$. As $\rho(Y_i')=2$, 
$-(K_{Y_i'}+{\Delta}_{Y_i'})$ is nef globally. This is a contradiction because 
the pushdown of $K_{Y_i'}+{\Delta}_{Y_i'}$ to $X_i''$ is ample.

\emph{Step 5.}
Now replace $({X_i'},B_i')$ with $({X_i''},B_i'')$ and repeat the process of Step 4, $m$ times. 
By the last paragraph the new components of $B_i'$ that 
appear in the process are not contracted again. So 
we may assume that we have at least $m$ components of $B_i'$ with coefficients $\ge 1-\epsilon$. 
Let $x_i'$ be the image of the exceptional divisor of $Y_i'\to X_i'$ and 
let $x_i''$ be the image of the exceptional divisor of $Y_i'\to X_i''$. 
Also let $m_i'$ be the number of those components of $B_i'$ with coefficient 
$\ge 1-\epsilon$ and passing through $x_i'$. Define $m_i''$ similarly. Since 
$\rho(Y_i')=2$, each component of $B_{Y_i'}$ intersects the exceptional 
divisor of $Y_i'\to X_i'$ or the exceptional divisor of $Y_i'\to X_i''$. Therefore,  
$m_i'+m_i''\ge m$. 

Finally by Lemma \ref{l-bnd-comps-surfaces} both $m_i'$ and $m_i''$ 
are bounded hence $m$ is also bounded. This means that after finitely many times applying the 
process of Step 4, we can assume there is $\epsilon>0$ such that $X_i'$ is $\epsilon$-lc for every $i$, 
and then apply boundedness of such $X_i'$ [\ref{Alexeev}]. 
\end{proof}

\subsection{$3$-folds}

\begin{proof}(of Theorem \ref{t-ACC})
If the theorem does not hold, then there is a sequence $(X_i,B_i)$ of lc pairs of 
dimension $3$ over $k$ and $\R$-Cartier divisors $M_i\ge 0$ such that the coefficients of 
$B_i$ are in $\Lambda$, the coefficients of $M_i$ are in $\Gamma$ but such that 
the $t_i:=\lct(M_i,X_i,B_i)$ form a strictly increasing sequence of numbers.
We may assume that each $(X_i,\Delta_i:=B_i+t_iM_i)$ has an lc centre of dimension $\le 1$ 
contained in $\Supp M_i$.  
Let $(Y_i,\Delta_{Y_i})$ be a $\Q$-factorial dlt model of 
$(X_i,\Delta_i)$ such that there is an exceptional divisor on $Y_i$ 
mapping onto an lc centre inside $\Supp M_i$. Such $Y_i$ exist by Lemma \ref{l-extraction}. 

There is a prime exceptional divisor $E_i$ of $Y_i\to X_i$ such that it intersects the 
birational transform of $M_i$ and that it maps into $\Supp M_i$. 
Note that $E_i$ is normal by Lemma \ref{l-plt-normal}.
Let $E_i\to Z_i$ be the contraction induced by $E_i\to X_i$. Now by adjunction define 
$K_{E_i}+\Delta_{E_i}=(K_{Y_i}+\Delta_{Y_i})|_{E_i}$. Then the set of all the coefficients of the 
horizontal$/Z_i$ components of the $\Delta_{E_i}$ satisfies DCC  
but  not ACC, by Proposition \ref{p-adjunction-DCC}. 
This contradicts Proposition \ref{p-ACC-global}.

\end{proof}

\section{Non-big log divisors: proof of \ref{t-aug-b-non-big}}\label{s-numerical}

\begin{lem}\label{l-movable-curve}
Let $X$ be a normal projective variety of dimension $d$ over an algebraically closed field 
(of any characteristic). Let $A$ an ample $\R$-divisor and 
$P$ a nef $\R$-divisor with $P^d=0$. Then for any $\epsilon>0$, there exist  
$\delta\in [0,\epsilon]$ and a very ample divisor $H$ such that $(P-\delta A)\cdot H^{d-1}=0$.
\end{lem}
\begin{proof}
First we show that there is an ample divisor $H$ such that $(P-\epsilon A)\cdot H^{d-1}<0$. 
Put $r(\tau):=(P-\epsilon A)(P+\tau A)^{d-1}$. 
Then
$$
r(\tau)=(P-\epsilon A)(P^{d-1}+a_{d-2}\tau P^{d-2}A+\dots+a_1\tau^{d-2}PA^{d-2}+\tau^{d-1}A^{d-1})
$$
 where the $a_i>0$ depend only on $d$. Put $a_{d-1}=a_0=1$, $a_{-1}=0$, and let $n$ be the smallest 
 integer such that $P^{d-n}A^{n}\neq 0$. Then we can write  
$$
r(\tau)=\sum_{i=0}^{d-1} (a_{i-1}\tau^{d-i}-\epsilon a_i\tau^{d-i-1}) P^iA^{d-i}
$$
from which we get 
$$
r(\tau)=\sum_{i=0}^{d-n} (a_{i-1}\tau^{d-i}-\epsilon a_i\tau^{d-i-1}) P^iA^{d-i}
$$
hence 
$$
\frac{r(\tau)}{\tau^{n-1}}=(a_{d-n-1}\tau-\epsilon a_{d-n}) P^{d-n}A^{n}+\tau s(\tau)
$$
for some polynomial function $s(\tau)$. Now if $\tau>0$ is sufficiently small it is clear 
that the right hand side is negative hence $r(\tau)<0$. 
  
Choose $\tau>0$ so that $r(\tau)<0$. Since $P+\tau A$ is ample and ampleness 
is an open condition, there is an ample $\Q$-divisor $H$ close to $P+\tau A$ 
such that $(P-\epsilon A)\cdot H^{d-1}<0$. By replacing $H$ with a multiple we 
can assume that $H$ is very ample. Since $P\cdot H^{d-1}\ge 0$ by the nefness of $P$, 
it is then obvious that there is some 
$\delta\in [0,\epsilon]$ such that $(P-\delta A)\cdot H^{d-1}=0$.\\
\end{proof}

\begin{proof}(of Theorem \ref{t-aug-b-non-big})
Assume that $D^d=0$. By replacing $A$ we may assume that it is ample.
Fix $\alpha>0$. By Lemma \ref{l-movable-curve}, there exist a number $t$ sufficiently close to $1$ 
(possibly equal to $1$) and a very ample divisor $H$ such that  
$$
(K_X+B+t(A+\alpha D))\cdot H^{d-1}=0
$$ 
Now we can view $H^{d-1}$ as a $1$-cycle on $X$. For each point $x\in X$, there is an effective 
$1$-cycle $C_x$ whose class is the same as $H^{d-1}$ and such that $x\in C_x$. Since 
$H$ is very ample, we may assume that $C_x$ is irreducible and that it is 
inside the smooth locus of $X$ for general $x$.
In particular, we have  
$$
(K_X+B+t(A+\alpha D))\cdot C_x= 0
$$ 

Pick a general $x\in X$ and let $C_x$ be the curve mentioned above. 
Since $B$ is effective and $A+\alpha D$ is ample, we get   
$K_X\cdot C_x<0$. 
Thus by Koll\'ar [\ref{kollar}, Chapter II, Theorem 5.8],  
there is a rational curve $L_x$ passing through $x$  such that 
$$
0<A\cdot L_x\le (A+\alpha D)\cdot L_x\le (2d) \frac{(A+\alpha D)\cdot C_x}{-K_X\cdot C_x}
$$
$$
=\frac{2d}{t} (1+\frac{B\cdot C_x}{K_X\cdot C_x})\le \frac{2d}{t}<3d
$$ 
because $K_X\cdot C_x<0$, $B\cdot C_x\ge 0$, and $t$ is sufficiently close to $1$. 
Note that although $K_X$ and $B$ need not be $\R$-Cartier, the intersection 
numbers still make sense since $C_x$ is inside the smooth locus of $X$.

As $A$ is ample and $A\cdot L_x\le 3d$, we can 
assume that such $L_x$ (for general $x$) belong to a bounded family $\mathcal{L}$ of curves 
on $X$ (independent of the choice of $t,\alpha$). 
Therefore there are only
finitely many possibilities for the intersection numbers $D\cdot L_x$.  
If we choose $\alpha$ sufficiently large, then the inequality $(A+\alpha D)\cdot L_x\le 3d$ 
implies $D\cdot L_x=0$ 
and so we get the desired family.\\
\end{proof}


\vspace{2cm}

\flushleft{DPMMS}, Centre for Mathematical Sciences,\\
Cambridge University,\\
Wilberforce Road,\\
Cambridge, CB3 0WB,\\
UK\\
email: c.birkar@dpmms.cam.ac.uk\\

\end{document}